\documentclass[11pt]{article}
\usepackage{amsmath}
\usepackage{amssymb,amsbsy,amsthm,booktabs,array,mathrsfs}
\theoremstyle{plain}
\usepackage{graphicx}
\usepackage[dvipsnames]{xcolor}
\usepackage{bbm}
\usepackage{mathtools}
\usepackage{environ}
\usepackage{xparse}
\usepackage{etoolbox}
\usepackage{tikz}
\allowdisplaybreaks
\usepackage[linktocpage,colorlinks,linkcolor=purple,anchorcolor=blue,citecolor=purple,urlcolor=black,pagebackref]{hyperref}
\usepackage[authoryear,round]{natbib}

\usepackage{enumerate}
\usepackage{cases}
\usepackage{thm-restate}

\usepackage[margin=1in]{geometry}
\usepackage{pdfsync}

\allowdisplaybreaks[1]

\numberwithin{equation}{section}

\DeclareMathOperator{\R}{\mathbb{R}} 
\DeclareMathOperator{\N}{\mathbb{N}} 

\newcommand{\p}{\mathbb{P}} 
\renewcommand{\P}{\mathbb{P}} 
\newcommand{\E}{\mathbb{E}} 
\newcommand{\eps}{\varepsilon} 
\DeclareMathOperator{\Var}{Var} 
\DeclareMathOperator{\Cov}{Cov} 
\newcommand{\TV}{\mathrm{TV}} 

\newcommand{\wh}{\widehat}  
\newcommand{\wt}{\widetilde} 

\newcommand{\UA}{\mathrm{UA}} 

\newcommand{\CUA}{\mathrm{CUA}} 
\newcommand{\alphaCUA}{\alpha{\text -}\mathrm{CUA}} 
\newcommand{\alphaprimeCUA}{\alpha^\prime{\text -}\mathrm{CUA}} 

\def\cA{{\mathcal A}}
\def\cB{{\mathcal B}}
\def\cC{{\mathcal C}}
\def\cD{{\mathcal D}}

\def\cF{{\mathcal F}}

\def\cN{{\mathcal N}}

\def\cR{{\mathcal R}}
\def\cS{{\mathcal S}}

\newtheorem{theorem}{Theorem}[section]
\newtheorem*{theorem*}{Theorem}
\newtheorem{lemma}[theorem]{Lemma}
\newtheorem*{lemma*}{Lemma}
\newtheorem{proposition}[theorem]{Proposition}
\newtheorem*{proposition*}{Proposition}

\newtheorem{corollary}[theorem]{Corollary}

\newtheorem{definition}[theorem]{Definition}

\def\tree{T}

\newcommand{\st}[1]{^{#1}}

\usepackage{enumitem}
\usepackage[nottoc,numbib]{tocbibind}

\def\be#1{\begin{equation*}#1\end{equation*}}
\def\ben#1{\begin{equation}#1\end{equation}}

\def\besn#1{\begin{equation}\begin{split}#1\end{split}\end{equation}}

\def\ba#1{\begin{align*}#1\end{align*}}
\def\ban#1{\begin{align}#1\end{align}}

\makeatletter

\def\given{\mskip 0.5mu plus 0.25mu\vert\mskip 0.5mu plus 0.15mu}
\newcounter{@bracketlevel}
\def\@bracketfactory#1#2#3#4#5#6{
\expandafter\def\csname#1\endcsname##1{%
\addtocounter{@bracketlevel}{1}%
\global\expandafter\let\csname @middummy\alph{@bracketlevel}\endcsname\given%
\global\def\given{\mskip#5\csname#4\endcsname\vert\mskip#6}\csname#4l\endcsname#2##1\csname#4r\endcsname#3%
\global\expandafter\let\expandafter\given\csname @middummy\alph{@bracketlevel}\endcsname
\addtocounter{@bracketlevel}{-1}}%
}
\def\bracketfactory#1#2#3{%
\@bracketfactory{#1}{#2}{#3}{relax}{0.5mu plus 0.25mu}{0.5mu plus 0.15mu}
\@bracketfactory{b#1}{#2}{#3}{big}{1mu plus 0.25mu minus 0.25mu}{0.6mu plus 0.15mu minus 0.15mu}
\@bracketfactory{bb#1}{#2}{#3}{Big}{2.4mu plus 0.8mu minus 0.8mu}{1.8mu plus 0.6mu minus 0.6mu}
\@bracketfactory{bbb#1}{#2}{#3}{bigg}{3.2mu plus 1mu minus 1mu}{2.4mu plus 0.75mu minus 0.75mu}
\@bracketfactory{bbbb#1}{#2}{#3}{Bigg}{4mu plus 1mu minus 1mu}{3mu plus 0.75mu minus 0.75mu}
}
\bracketfactory{clc}{\lbrace}{\rbrace}
\bracketfactory{clr}{(}{)}
\bracketfactory{cls}{[}{]}
\bracketfactory{abs}{\lvert}{\rvert}
\bracketfactory{norm}{\Vert}{\Vert}
\bracketfactory{floor}{\lfloor}{\rfloor}
\bracketfactory{ceil}{\lceil}{\rceil}
\bracketfactory{angle}{\langle}{\rangle}

\title{Correlated uniform attachment trees}
\author{Johannes B\"aumler\thanks{Mathematical Institute, University of Koblenz; \url{jbaeumler@uni-koblenz.de}}
\and 
Mikl\'os Z. R\'acz\thanks{Department of Statistics and Data Science and Department of Computer Science, Northwestern University; \url{miklos.racz@northwestern.edu}}
\and 
Nathan Ross\thanks{School of Mathematics and Statistics, University of Melbourne; \url{nathan.ross@unimelb.edu.au}}
\and
Anirudh Sridhar\thanks{Helen and John C.\ Hartmann Department of Electrical and Computer Engineering, NJIT; \url{anirudh.sridhar@njit.edu}}}
\date{\today}

\begin{document}

\maketitle


\begin{abstract}
We introduce and study a new model of correlated uniform attachment (UA) trees, where correlation is sprinkled throughout the time evolution of the process. In this model, two UA trees are grown in parallel, and at each time step a new node is added to each tree, with an edge between it and a uniformly chosen existing vertex in the respective tree. The two choices of attachment are correlated: with probability $\alpha$, the edges attach to nodes with the same time label in both trees, and with probability $1-\alpha$, the choices are made independently. 
We study fundamental detection and estimation questions for this model, given two \emph{unlabeled} trees. 
In our main result, we construct a consistent estimator of the correlation parameter $\alpha$, as the size of the trees goes to infinity. 

The construction of our statistic relies on two key ideas. First, we use Jordan centrality to identify  subsets of vertices of each tree whose intersection has a  sufficient number of common early vertices. 
The second idea is that, 
across multiple time scales, 
it is possible to approximately determine the labels of vertices that have attached to these early vertices, using the sizes of fringe subtrees. 
Our analysis includes novel quantitative bounds on the fraction of early vertices that remain central, which are of independent interest in the network archaeology literature. 
\end{abstract}




\section{Introduction and overview} \label{sec:intro} 

Networks arise naturally in a number of fields, such as protein-protein networks in biology, social networks in sociology, and power grids in engineering, just to name a few. 
Motivated by understanding  such data, there has been   an enormous effort over the last 
decades 
to define and study the properties of \emph{random networks} for the purpose of inference.
Modern introductions to this literature include the texts \cite{Barbour2026}, \cite{Newman2018}, \cite{Hofstad2017, Hofstad2024}. 

Increasingly, important settings involve not only a single network but multiple networks which are correlated. 
Motivating applications include data privacy in social networks~\citep{Narayanan2009} and inferring function in biological networks \citep*{Singh2008}. 
Detecting correlation among the networks and estimating aspects of the correlated structure, such as latent matchings of the vertices across networks, are central questions. 
Inspired by these applications and questions, the past decade has seen an explosion of exciting work on correlated random graphs. 
Initially, these focused on the simplest such model, namely correlated Erd\H{o}s--R\'enyi random graphs, introduced by \cite{pedarsani2011privacy}, which already exhibits interesting phenomena and presents significant technical challenges. 
By now, there is a fairly comprehensive understanding of the main questions for this model, including information-theoretic limits and algorithms for detection and estimation, due to the combined efforts of many works, including
\cite*{cullina2016improved,cullina2018exact},
\cite*{ding2021efficient},
\cite*{ganassali2020tree},
\cite*{ganassali2021impossibility} 
\cite*{Ganassali2024a},
\cite*{wu2022settling,Wu2023}, 
\cite*{mao2023exact}, 
\cite*{mao2023random,Mao2024}, 
\cite*{Ganassali2024},
\cite*{Ding2023,Ding2023a}, 
\cite*{ding2025low}.

An important research direction is to broaden the understanding of these questions by studying them on more realistic models of correlated random graphs, beyond Erd\H{o}s--R\'enyi. 
While correlated stochastic block models have been explored substantially 
(\cite{lyzinski2015spectral}, 
\cite*{onaran2016optimal}, 
\cite*{lyzinski2018information}, 
\cite*{racz2021correlated}, 
\cite*{GRS22}, 
\cite*{yang2023efficient}, 
\cite*{yang2023graph}, 
\cite*{chai2024efficient}, 
\cite*{racz2024harnessing}, 
\cite*{Chen2026,chen2026detecting}), 
models with other structure have received relatively little attention so far. 
The few examples include 
correlated randomly growing graphs~\citep{RS22AAP}, 
correlated random geometric graphs~\citep*{wang2022random}, 
correlated inhomogeneous random graphs~\citep*{Racz2023,Ding2025,ameen2024aligning}, 
and models with power-law degree distributions~\citep*{korula2014efficient,yu2021power}.

In this paper, we contribute to this line of research by introducing and studying a novel model of correlated randomly growing graphs. 
The main feature of this model is that correlation is sprinkled throughout the time evolution of the process. 
This feature makes the model substantially different from the one studied by~\cite{RS22AAP}, which is a change-point model where two graphs initially grow together and subsequently grow independently. 
We focus our attention on \emph{correlated uniform attachment (UA) trees}, which is arguably the simplest setting of 
such a model. 

We study fundamental detection and estimation questions in this new model of correlated UA trees. 
Our first result shows that correlated UA trees can be distinguished from independent ones with probability tending to one as the size of the trees goes to infinity; see Theorem~\ref{thm:detection}. 
Our main result goes a step further and shows that the level of correlation 
can be consistently estimated asymptotically; see Theorem~\ref{thm:estimation}. 

In the spirit of other results in this area, our proofs rely on (rough) vertex matching, both for early-arriving vertices and for late-arriving vertices. 
Approximately matching early vertices is closely connected to a thread of recent work on \emph{network archaeology}, 
which focuses on recovering parts of the past history (e.g., the root or early vertices) of a growing tree/graph given an unlabeled snapshot at a particular time. 
While the network archaeology literature has developed significantly over the past decade
(see, e.g., 
\citet*{BMR15},
\citet*{CDKM15},
\citet*{bubeck2017finding},
\citet*{BEMR17},
\citet*{Jog2017,Jog2018},
\citet*{Lugosi2019},
\citet*{Devroye2019},
\citet*{Banerjee2022},
\citet*{Briend2023},
\citet*{contat2024eve},
\citet*{Briend2025},
\citet*{addario2024optimal}), 
the existing results are insufficient for our purposes (see Section~\ref{sec:earlyvert} for details). 
Thus, as a key component in our proofs, we prove a novel result that is of independent interest in network archaeology: 
we quantify in UA trees the size of the intersection of the earliest vertices and the most central vertices (measured according to Jordan centrality); see Theorem~\ref{thm:centrality}.

\subsection{The model and results}
An unlabeled uniform attachment ($\UA$) tree sequence $(\tree_n)_{n\geq0}$ is constructed sequentially by adding one vertex with an edge at each time step. Initially, $\tree_0$ consists of a single node, and, given $\tree_n$, $\tree_{n+1}$ is constructed by adding a single node with an edge connected to a node of $\tree_{n}$ chosen uniformly at random. Distributionally we write $\tree_n\sim\UA(n)$. 
It is also natural to consider a \emph{labeled} model, where the label of a vertex is the time of its arrival: the initial vertex has label $0$ and the vertex added at step $n$ is labeled $n$.

We introduce and study a model of two \emph{correlated} $\UA$ tree sequences $(T_n\st{1},T_n\st{2})_{n\geq0}$. 
The model has a single parameter $\alpha \in [0,1]$, which measures the amount of dependence between the two trees. 
The two trees are grown sequentially together, 
and marginally they are both UA trees, 
but the edge choices at each time step are correlated. 
Specifically, 
at each time step, 
with probability $\alpha$ the edge choices are \emph{aligned}, 
and with probability $1-\alpha$ they are \emph{independent},
defined as follows: 
\begin{itemize}
\item \textbf{Aligned choice:} 
the two new nodes connect to a node in their respective trees with the \emph{same} time label, chosen uniformly at random. 
That is, at time step $n$, choose $i \in \{0,1,\ldots, n-1\}$ uniformly at random, and in both trees node $n$ connects to node $i$.
\item \textbf{Independent choice:} 
the two new nodes make edge choices in their respective trees independently. 
That is, at time step $n$, choose $i_{1}, i_{2} \in \{0,1,\ldots, n-1\}$ independently and uniformly at random, 
in the first tree node $n$ connects to node $i_{1}$ 
and in the second tree node $n$ connects to node $i_{2}$. 
\end{itemize}

\begin{figure}[htbp]
\centering
\begin{minipage}{0.45\textwidth}
\centering
\includegraphics[width=\linewidth]{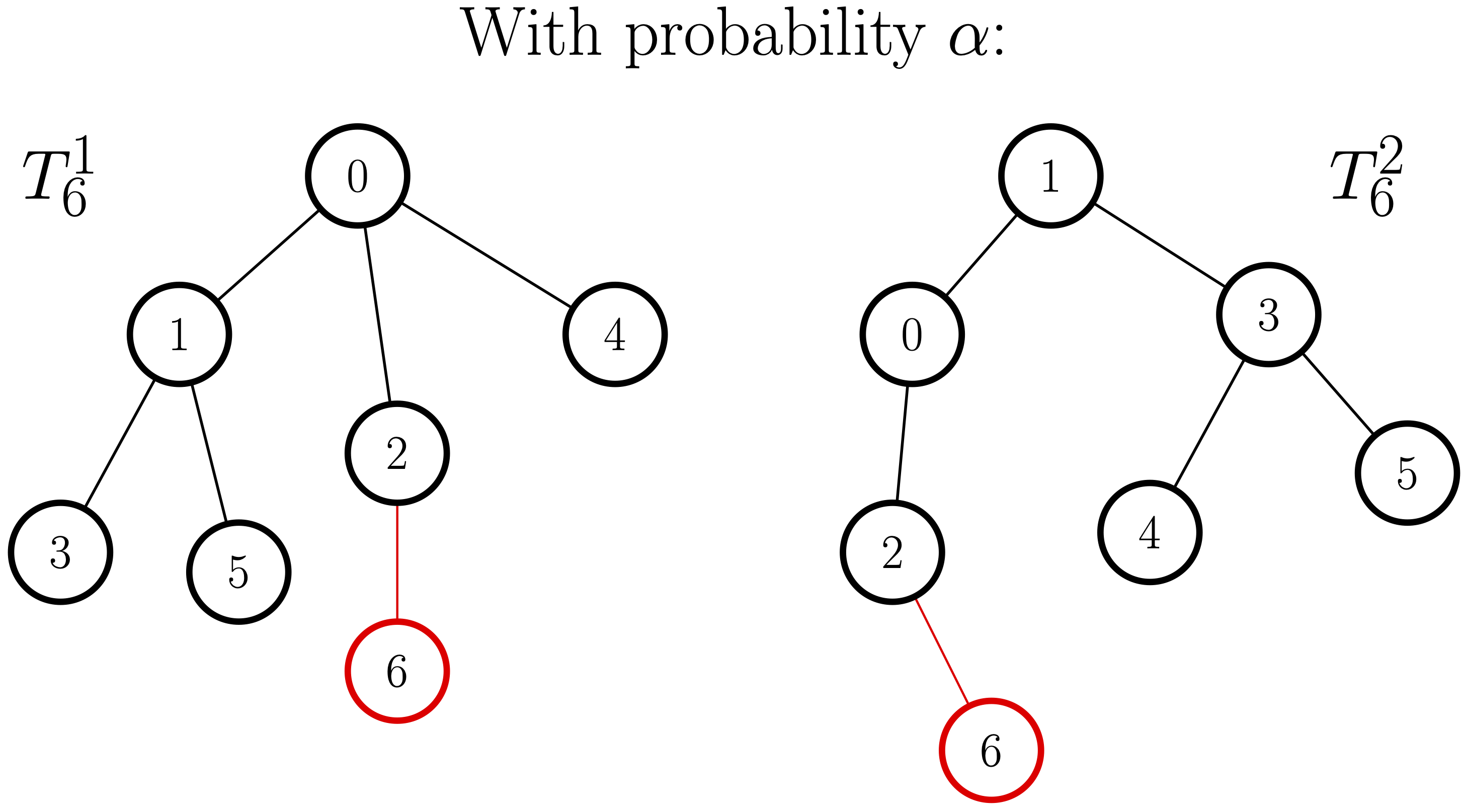}
\end{minipage}
\hfill 
\begin{minipage}{0.45\textwidth}
\centering
\includegraphics[width=\linewidth]{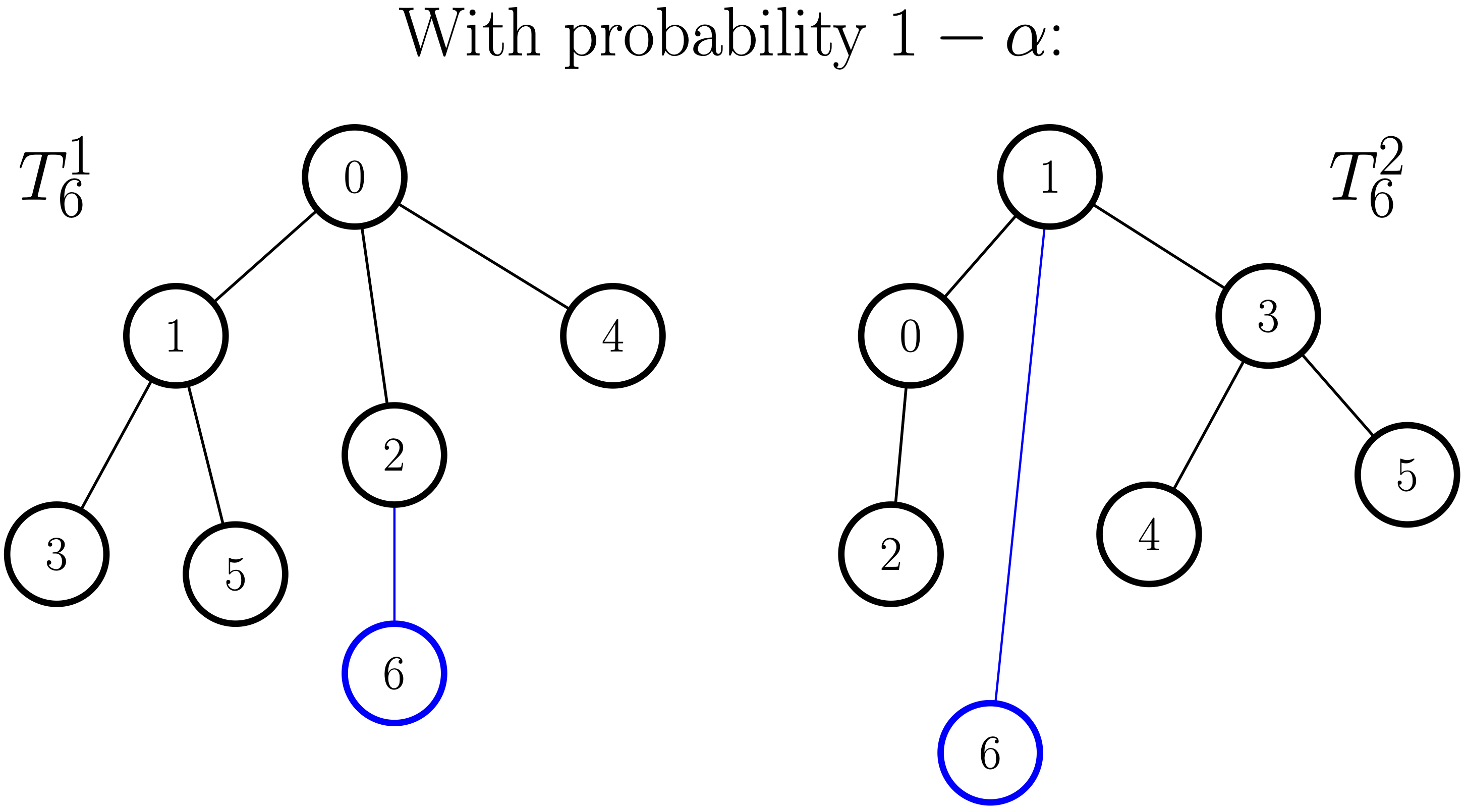}
\end{minipage}
\caption{An illustration of one step in the model. 
Before this step, the two trees with nodes $\{0,1,\ldots,5\}$ are in black. 
Subsequently, a new node, labeled $6$, is added to both trees. 
The edge choices in the two trees are correlated, as follows.
With probability $\alpha$, the two new nodes connect to a node in their respective trees with the \emph{same} time label, chosen uniformly at random (\emph{aligned choice}, illustrated on the left). 
With probability $1-\alpha$, the two new nodes make edge choices in their respective trees independently (\emph{independent choice}, illustrated on the right).}
\label{fig:model}
\end{figure}

The model is illustrated in Figure~\ref{fig:model}. 
The time labels are an essential part of the definition of this model of correlated UA trees. 
Marginally, $T_{n}^{1}$ and $T_{n}^{2}$ are both uniform attachment trees, and 
they are correlated through the ``aligned choice'' steps in their time evolution, 
where the time labels play a central role. 
Nonetheless, since 
our focus is on inference problems given two such trees \emph{without} knowledge of the time labels, 
in what follows we write 
$\left( T_{n}^{1}, T_{n}^{2} \right) \sim \alphaCUA \left( n \right)$ 
for a pair of \emph{unlabeled} $\UA$ trees constructed in this way. 
When $\alpha = 0$, the two trees are independent---we also write $\UA(n)^{\otimes 2}$ for $0{\text-}\CUA(n)$---and when $\alpha = 1$, we have  $T_{n}^{1} = T_{n}^{2}$. Thus  $\alpha \in [0,1]$ quantifies the amount of dependence between the two trees. 

Given this model, a natural first question is whether it is possible to asymptotically distinguish between a pair of trees generated according to $\alphaCUA (n)$ for $\alpha>0$ and an independent pair of $\UA(n)$ trees. 
An inherent difficulty in this problem (and subsequent estimation problems) is that the data consists of an \emph{unlabeled} pair of trees; while the ``aligned choice'' step aligns the choices of nodes in terms of \emph{labels}, such a step may actually make the pair of \emph{unlabeled} trees go \emph{further apart}.\footnote{As an example, consider $n=2$, let $T_{2}^{1}$ and $T_{2}^{2}$ both be paths on the node set $\{0,1,2\}$, with the central node in $T_{2}^{1}$ having time label $0$ and the central node in $T_{2}^{2}$ having time label $1$. 
Now suppose that the next step is an ``aligned choice'' step, with node $3$ attaching to node $0$ in both trees. Then $T_{3}^{1}$ is a star on four nodes and $T_{3}^{2}$ is a path on four nodes, 
so the two trees are not isomorphic after the ``aligned choice'' step, despite being isomorphic before it.}
Despite this, we are able to distinguish correlated trees from independent ones. 
To state this result, recall that the total variation distance between two probability measures $\textsf{P}$ and $\textsf{Q}$ on a discrete set $\mathcal{X}$ is defined by
\ben{\label{eq:dtvdef}
\TV(\textsf{P},\textsf{Q}) := \sup_{A \subseteq \mathcal{X}} | \textsf{P}(A)-\textsf{Q}(A)|.
}

\begin{theorem}[Detecting correlation]\label{thm:detection}
For every $\alpha \in (0,1]$ we have that 
\begin{equation}\label{eq:TV_to_1}
\lim_{n \to \infty} \TV \left( \alphaCUA \left( n \right), \UA(n)^{\otimes 2} \right) = 1.
\end{equation}
\end{theorem}

We initially attempted to prove Theorem~\ref{thm:detection} using various simple statistics, such as counting leaves in the two trees. 
However, while such approaches are able to show that the $\TV$ distance in~\eqref{eq:TV_to_1} is bounded away from $0$ (see Appendix~\ref{sec:leaves} for the details for counting leaves), 
these approaches cannot show that the limiting $\TV$ distance is $1$. 
To show this, we devise a statistic that is inherently multiscale in nature. 
This statistic computes the (normalized) sum of the degrees of the $K_{j}$ most central vertices (measured according to Jordan centrality) in both trees, across a range of different values $\{K_{j}\}_{j=1}^{\ell}$ with $K_{j} = 2^{j} K$; 
see Section~\ref{sec:proofideas} for further discussion 
and Lemmas~\ref{lem:sep1} and~\ref{lem:sep2} 
for details. 
A key ingredient of our proof is a result that quantifies in UA trees the size of the intersection of the earliest vertices and the most central vertices. 
This is of independent interest in the network archaeology literature; see Section~\ref{sec:earlyvert} and Theorem~\ref{thm:centrality} below for details. 

We can go a step further and ask 
whether it is possible to asymptotically distinguish between two pairs of trees generated according to $\alphaCUA(n)$ and $\alphaprimeCUA(n)$ for $\alpha\not=\alpha'$. 
Our next result affirmatively answers this question by constructing a consistent estimator of~$\alpha$.

\begin{theorem}[Consistent estimation of $\alpha$]\label{thm:estimation}
There exists an estimator 
$\wh{\alpha}_{n} := \wh{\alpha}_{n}(T_{n}^{1}, T_{n}^{2})$ 
such that for every $\alpha \in [0,1]$, 
if $(T_{n}^{1}, T_{n}^{2}) \sim \alphaCUA(n)$, 
then $\wh{\alpha}_{n} \to \alpha$ in probability as $n \to \infty$.
\end{theorem}

The construction of the estimator $\wh{\alpha}_{n}$ and the proof of Theorem~\ref{thm:estimation} build on the statistics and proof of Theorem~\ref{thm:detection}, but are more involved. 
In particular, $\wh{\alpha}_{n}$ is also multiscale in nature and relies on estimating early vertices using the most central vertices. 
However, instead of computing degrees, the estimator $\wh{\alpha}_{n}$ takes a more fine-grained approach and computes, for central vertices, the number of neighbors whose fringe subtrees have sizes that fall in appropriate intervals. 
Effectively, the sizes of fringe subtrees allow us to approximately estimate the time labels of late-arriving vertices (at least in aggregate), and doing this across multiple time scales allows to estimate $\alpha$.
We note that the estimator $\wh{\alpha}_{n}$ is computable in polynomial time in $n$.
We refer to Section~\ref{sec:detalp} for further details. 
The existence of this consistent estimator immediately implies the  following corollary. 

\begin{corollary}[Distinguishing different levels of correlation]\label{cor:main}
For every $\alpha,\alpha' \in [0,1]$ with~$\alpha \neq \alpha'$, we have that 
\[
\lim_{n \to \infty} \TV \left( \alphaCUA \left( n \right), \alphaprimeCUA(n) \right) = 1.
\]
\end{corollary}

While Corollary~\ref{cor:main} implies Theorem~\ref{thm:detection}, we state it separately because the proof of Theorem~\ref{thm:detection} is fundamentally different and simpler, and moreover it provides insights into the structure of UA trees in general, including the location of their most central vertices, which is a further contribution of the paper. We continue this introduction by discussing this latter contribution, followed by outlines of the proofs of Theorems~\ref{thm:detection} and~\ref{thm:estimation} in Section~\ref{sec:proofideas}, then discussing open problems in Section~\ref{sec:openprobs}, and finally giving the plan of the rest of the paper in Section~\ref{sec:outline}.

\subsection{Finding early vertices}\label{sec:earlyvert}

A key element of our proofs is identifying collections of vertices of each tree that are likely to contain early vertices and have many labels in common. Statistics related to such vertices can then give information about the correlation between the two trees.
The problem of finding the root and early vertices of unlabeled randomly growing trees, including UA, has received significant attention in recent years, catalyzed by \citet*{bubeck2017finding}.  
In particular, closest to our paper are the works of 
\cite{Lugosi2019} and \cite{Devroye2019}, 
who study the problem of finding the \emph{seed}, that is, the tree induced by the earliest vertices. 
However, for reasons discussed below, the results 
therein
are insufficient for our purposes, 
so we derive a related result that is of independent interest in network archaeology. 
Specifically, we quantify in UA trees the size of the intersection of the earliest vertices and the most central vertices; see Theorem~\ref{thm:centrality} below. 

Like many related works, 
we use the following notion of centrality to identify early vertices.
\begin{definition}[Centrality]\label{def:centrality}
For a vertex $v$ of a tree $\tree$, let $\psi_\tree(v)$ denote the size of the maximal component after removing all edges connected to $v$, and call $\psi_\tree(v)$
the \emph{centrality} of vertex $v$.
\end{definition}
For a tree $\tree$, we refer to the  $K$  vertices of $\tree$ with the \emph{smallest} centrality (with ties broken uniformly at random) as the $K$ \emph{most central vertices} of $\tree$. 
(This notion of centrality is referred to as \emph{Jordan centrality} in many network archaeology works, but that terminology is also used for a different notion of centrality---in terms of graph distances---in other areas.) 
The recent paper \citet*{Josifov2026} discusses properties of the most central vertices (for this centrality measure and other common ones) in the context of root finding in UA trees, among others; see also 
\cite{Jog2017,Jog2018}, 
\cite{Banerjee2022}, 
and references therein.

Here, we quantify the overlap of the most central vertices with the earliest vertices. 
\begin{theorem}\label{thm:centrality}
Let $T_{n} \sim \UA(n)$ and let $A_{K}(n)$ be the set of the $K$ most central vertices of~$T_{n}$. Fix $c > 1$, $\delta > 0$,
and set 
\[
c^{\prime} := c \left(1 - e^{-1/c} \right) - \delta.
\]
Then there exists a constant $C < \infty$ that depends only on $c$ and $\delta$ such that for every 
$\eps > 0$, 
the following holds for all $K \geq (C/\eps) \log(1/\eps)$ 
and all large enough $n \geq n_{0}(K,\delta, c )$: 
\[
\p \left( \left| A_{\ceil{c K}}(n) \cap \{ 0, 1, \ldots, K-1 \} \right| \geq c^{\prime} \,   K \right) 
\geq 1 - \eps.
\]
\end{theorem}
The theorem says that for large enough $K$, with probability close to one, the proportion of the first $K$ vertices that are also among the $cK$ most central is at least    $c(1-e^{-1/c})-\delta$ for any $\delta>0$. This proportion is optimal with our proof, but it is an interesting question to study the behavior of $|A_{K}(n) \cap \{0,\ldots,K-1\}|/K$ as $K\to\infty$.
Figure~\ref{fig:centrality} is a simulation illustrating the theorem for $K=50$ and $c=2$, where the overlap between the $50$ first arriving vertices and the $100$ most-central vertices is 41. Further simulations suggest the mean overlap is around 40 with a standard deviation of a bit more than $2$, which aligns well with the threshold $2(1- e^{-1/2}) \times 50 = 39.35$ given by Theorem~\ref{thm:centrality}.

\begin{figure}[ht!]
    \centering
    \includegraphics[width=0.86\linewidth]{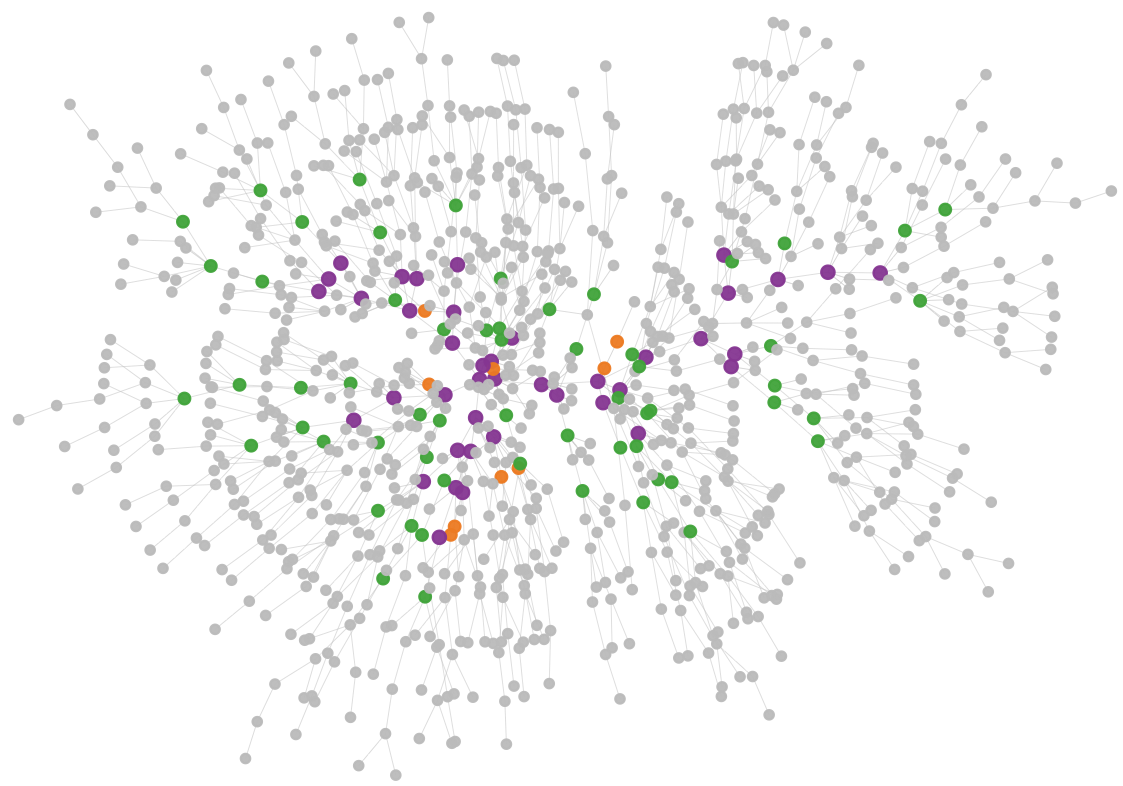}
    \caption{A simulation of a UA tree with 1000 vertices illustrating Theorem~\ref{thm:centrality} with $K=50$ and $c=2$. The 41 purple nodes are those within both the 50 first arrivals and the 100 most central nodes, orange are the 9 remaining of the 50 first arrivals, and green are the 59 remaining of the 100 most central.}
    \label{fig:centrality}
\end{figure}

We use Theorem~\ref{thm:centrality} to derive the following corollary that gives a lower bound on the overlap of the $K$-centrality sets of any two UA trees for large $K$; 
this corollary is a key component of our subsequent proof for detecting correlation. 
 \begin{corollary} \label{cor:intersection}
     Let 
     $(T_n^1,T_n^2)_{n\geq0}$ be a sequence of pairs of UA trees with any joint distribution, and denote their vertex-labeled centrality sets by $(A_r\st{1}(n),A_r\st{2}(n))$, $r\in \N$.
       Then there exists a constant $C<\infty$ such that for all $K$ and for all large enough $n\geq n_{0}(K)$,
    \begin{align*}
        \p\left( \left|A_K\st{1}(n) \cap A_K\st{2}(n) \right| \geq 0.2 K \right) \geq 1 - \frac{C \log(K)}{K}.
    \end{align*}
\end{corollary}

Let us now compare our results with those of~\cite{Lugosi2019} and~\cite{Devroye2019}, 
who study finding the seed in UA trees. 
There are two main differences between these works and ours. 
First, most of the results in~\cite{Lugosi2019} and~\cite{Devroye2019} 
focus on settings where the seed is either a known fixed graph (e.g., a path or a star), or certain properties of it are known (e.g., its number of leaves). 
Our interest is in the setting where the seed is itself a UA tree, so these results do not apply. 
The exception is \cite[Theorem~1.5]{Lugosi2019}, 
which shows that there exists a constant $c > 0$ such that for large~$K$, with probability close to one, 
the $c K / \log K$ most central vertices 
are all among the $K$ earliest vertices. 
However, for our purposes we would like to find $\Omega(K)$ of the $K$ earliest vertices (in fact, to obtain the linear overlap in $K$ in Corollary~\ref{cor:intersection} by a union bound, without any assumptions on the joint distribution of $(T_{n}^{1}, T_{n}^{2})$, we want to find at least $(1/2+\eps)K$ of the earliest $K$ vertices), so the aforementioned result is insufficient for our application.

This brings us to the second main difference between these works and ours, 
which is more conceptual. 
\cite{Lugosi2019} and~\cite{Devroye2019} focus on seed-finding algorithms that output, with probability close to one, a set that either contains the seed (i.e., it must contain all seed nodes) or is contained in the seed (i.e., it cannot contain nodes outside of the seed). 
In contrast, our goal is more relaxed, allowing both types of errors: we are fine with outputting a set which does not contain all of the earliest vertices, and which contains vertices that are not early. 
The upside of this flexibility is that it allows us to find a large fraction of the $K$ earliest vertices, as detailed in Theorem~\ref{thm:centrality}. 
Focusing on the overlap between the earliest vertices and the most central vertices is natural in its own right, so Theorem~\ref{thm:centrality} may find other applications.

The proof of Theorem~\ref{thm:centrality} follows the strategy (used in broad strokes in other works studying centrality in UA trees, e.g., \cite{Jog2018}, \cite{Lugosi2019}) of relating
centrality to the sizes of so-called \emph{fringe subtrees} of early and late vertices; see Definitions~\ref{def:frst} and~\ref{def:delgd}, and Lemma~\ref{lem:centear}. For a~$\UA$ tree, fringe subtrees evolve as standard P\'olya urns started with one ball of the counted color, which we can understand in detail; see Section~\ref{sec:uadelgd}.

\subsection{The statistics and proof ideas}\label{sec:proofideas}

We now give a high level overview of the key ideas behind our proofs and statistics. Recall that the input to our statistics is a pair of \emph{unlabeled} trees $(T_{n}^{1}, T_{n}^{2})$. 
Imagine, however, that we actually \emph{knew} the identity of the initial vertex (labeled $0$) in both trees. 
We now explain how this can help with distinguishing the distributions $\alphaCUA(n)$ and $\UA(n)^{\otimes 2}$. 

Let $D_{n}^{i} := \deg_{T_{n}^{i}} (0)$ denote the degree of the initial vertex in $T_{n}^{i}$, for $i \in \{1,2\}$. 
A natural idea is to use the pair of degrees $(D_{n}^{1}, D_{n}^{2})$ as a statistic to distinguish the distributions $\alphaCUA(n)$ and $\UA(n)^{\otimes 2}$. 
Observe that we can write 
$D_{n}^{i} = \sum_{j=1}^{n} X_{j}^{i}$, 
where $X_{j}^{i}$ is the indicator random variable that is $1$ precisely when node $j$ attaches to node $0$ in the $i$-th tree. 
Since $\E[X_{j}^{i}] = 1/j$, 
we have that $\E[D_{n}^{i}] = \sum_{j=1}^{n} 1/j \approx \log n$, 
and in fact marginally $D_{n}^{i}$ is approximately Poisson distributed with mean $\log n$. 
Under the independent model $\UA(n)^{\otimes 2}$, 
we have that $D_{n}^{1}$ and $D_{n}^{2}$ are independent. 
On the other hand, under $\alphaCUA(n)$, the indicators $X_{j}^{1}$ and $X_{j}^{2}$ are dependent. 
Specifically, we have that 
$\p ( X_{j}^{1} = X_{j}^{2} = 1 ) = \alpha / j + (1- \alpha) / j^{2}$, 
which implies that 
$\Cov (D_{n}^{1}, D_{n}^{2}) \approx \alpha \log n$. 
Now consider the normalized (centered and scaled) degrees 
$Z_{n}^{i} := (D_{n}^{i} - \log n) / \sqrt{\log n}$ 
for $i \in \{1,2\}$.
By the preceding discussion, we have that the pair 
$(Z_{n}^{1}, Z_{n}^{2})$ 
is approximately standard Gaussian under $\UA(n)^{\otimes 2}$, 
while under $\alphaCUA(n)$ it is approximately Gaussian with 
mean $(0,0)$ 
and covariance matrix 
$\left(\begin{smallmatrix}
1 & \alpha \\
\alpha & 1
\end{smallmatrix}\right)$.
For every $\alpha \in (0,1)$, the total variation distance between these two Gaussian distributions is in $(0,1)$, 
so $\alphaCUA(n)$ and $\UA(n)^{\otimes 2}$ can be distinguished with positive probability. 

There are two main issues with this approach: 
(1) the identities of the initial vertices in the two trees are \emph{not known}, 
and 
(2) this only distinguishes with positive probability, not with probability close to one. 
The solution to the first issue is to turn to network archaeology and \emph{estimate} the earliest vertices in the two trees, 
while the solution to the second issue is to consider such statistics across multiple time scales. 
We dive into the details in Sections~\ref{sec:alpindsep} and~\ref{sec:detalp} below.

Before proceeding further, we note that 
observable statistics can only be in terms of an \emph{unlabeled} version of the tree~$\tree$, but many times probabilities of statistics are more conveniently written in terms of the \emph{labeled} version of the tree. 
To mitigate confusion around this point, we will write $v\in \tree$ to denote a vertex in the unlabeled tree (e.g., when indexing a sum), and we will use $\{0,\ldots, n\}$ to denote the vertices of labeled versions of the trees.

Both statistics used in our proofs of Theorems~\ref{thm:detection} and \ref{thm:estimation} are based on centrality sets defined using Definition~\ref{def:centrality}, and we denote the $K$ most central vertices of $T_n\st{i}$ by $A_{K}\st{i}(n)$, $1\leq K \leq n+1$.

\subsubsection{Distinguishing between independent and correlated models}\label{sec:alpindsep}
To construct our separating statistic 
to distinguish between independent and correlated pairs of UA trees, 
fix $\ell\in \N$ and integers $0< K_0 < K_1< \cdots < K_\ell$, and  define
\[
S_{K_j}\st{i}(n) := \sum_{v \in A_{K_j}\st{i}(n)} \deg_{T_{n}\st{i}}(v).
\]
In words, $S_{K_j}\st{i}(n)$ is the sum of the degrees of the $K_j$ vertices with largest centrality in $T_{n}\st{i}$. 

Our separating statistic is 
\ben{\label{eq:sepstat}
\frac{1}{(\ell+1)}\sum_{j=0}^\ell \frac{\bclr{S_{K_j}\st{1}(n)-S_{K_j}\st{2}(n)}^2}{K_j\log(n) },
}
for $K_j:= 2^j K$ and large enough $K$ and $\ell$.

To explain why this statistic is separating, first assume that the labels of the vertices in the sets $A_{K_j}\st{i}(n)$, $j=0,\ldots, \ell$, $i=1,2$, are \emph{known} (just like in the discussion at the beginning of Section~\ref{sec:proofideas}), and 
that these sets are \emph{fixed} for all large enough $n$.
It is then possible to estimate the mean and variance of the separating statistic \eqref{eq:sepstat}, and, if there is enough overlap between the sets $A_{K_j}\st{1}(n)$ and $A_{K_j}\st{2}(n)$, we expect the mean of the separating statistic to be smaller for the correlated case. This is because in the independent model, $\Cov(S_{K_j}\st{1}(n), S_{K_j}\st{2}(n)) = 0$,
while for the $\alphaCUA(n)$ model, $\Cov(S_{K_j}\st{1}(n), S_{K_j}\st{2}(n)) = \Theta \bclr{\alpha \abs{ A_{K_j}\st{1}(n)\cap A_{K_j}\st{2}(n)} \log (n) }$. 
Assuming further that $\abs{ A_{K_j}\st{1}(n) \cap A_{K_j}\st{2}(n)}\geq 0.2 K_j$, 
Lemma~\ref{lem:covariances} 
below shows that 
under the independent model, the mean of a modification of the separating statistic (which can be shown to be close to the original) is at least $(2-0.01\alpha)$; see~\eqref{eq:variance lem 1};
while under the correlated $\alphaCUA$ model,
the mean is no more than $(2-0.1\alpha)$; see~\eqref{eq:variance lem 2}. The lemma also gives  bounds on the variance under both models, and Chebyshev-type probability bounds on the probability of deviation from the mean in~\eqref{eq:sepstatprbind} and~\eqref{eq:sepstatprbcor}. 
Here we can also see that the scaling choice of $K_j = K 2^j$ is crucial for these deviation probabilities to be small for large $\ell$.
It is important to note that while these calculations rely on the vertex labels, the resulting bounds only rely on the assumption that the sets $A_{K_j}^i(n)$ are eventually fixed and have sufficient overlap for $i=1,2$. 

Given this setup, the high-level structure of the proof is to show that with good probability the sets
$A_{K_j}\st{i}(n)$, for $j=0,\ldots, \ell$ and $i=1,2$, 
are fixed for all $n$ large enough, 
and that there is enough overlap between them. That the sets are eventually fixed follows directly from Theorem~2 of~\cite{Jog2018} (recorded below as Theorem~\ref{theo:jog loh}),
where this property is referred to as \emph{persistence} and is used in root-finding algorithms; see also~\cite{Banerjee2022}.
That there is sufficient overlap is the content of Corollary~\ref{cor:intersection}, which follows from Theorem~\ref{thm:centrality}; see Section~\ref{sec:earlyvert}.

Finally, the reason that we are not able to use this statistic to 
distinguish between $\alphaCUA(n)$ and $\alphaprimeCUA(n)$ for all $\alpha, \alpha^\prime \in [0,1]$ with $\alpha \neq \alpha^\prime$ is that we would need a sharper version of Corollary~\ref{cor:intersection} with precise estimates on the sizes of the sets $|A_K\st{1}(n) \cap A_K\st{2}(n)|$ in terms of  $\alpha$, in order to control the mean of the separating statistic~\eqref{eq:sepstat}.

\subsubsection{Estimating the correlation parameter $\mathbf{\alpha}$}\label{sec:detalp}
For our more general result, we need some notation to describe how to define our consistent estimator of $\alpha$. For distinct vertices $u,v$ of a tree~$T_n^i$, we write $K_v \left( T_n^i \setminus \{u\} \right) $ for the subtree that contains $v$ when removing $u$ and all of its edges from~$T_n^i$. For a tree $T_n^i$ and vertices $u,v \in T_n^i$, we write $u \sim_i v$ if there is an edge between $u$ and $v$ in the tree $T_n^i$. For a vertex $u$ of a tree $\tree_n^i$, define
\ben{\label{eq:uhatdef}
    \widehat{U}_{\ell,j}^{T_n^i}(u) = \sum_{v\in \tree_n^i} \mathbbm{1} \left\{ v \sim_i u,  \left| K_v \left( T_n^i \setminus \{u\} \right) \right| \in \left( \frac{n}{e^j} , \frac{n}{e^\ell} \right] \right\},
}
as the number of neighbors $v$ of $u$ where the number of vertices of $K_v \left( T_n^i \setminus \{u\} \right) $ falls in the interval $(ne^{-j}, ne^{-\ell}]$. Note that for each vertex $u$, this statistic is observable from the unlabeled tree. We show below in Lemma~\ref{lem:sep3} that a consistent estimator of $2(1-\alpha)$ is 
\ben{\label{eq:consest}
 \min_{u_1 \in A_K^1(n), u_2 \in A_K^2(n)} \frac{1}{k^3} \sum_{j=1}^{k} \left( \widehat{U}_{jk^2,(j+1)k^2}^{T_n^1}(u_1) - \widehat{U}_{jk^2,(j+1)k^2}^{T_n^2}(u_2) \right)^2
}
for an appropriate choice of $K=K(n)$ and $k=k(n)$. 
To explain how this statistic is consistent, consider a vertex 
$u\in \{1,\ldots, M\}$ for some large but fixed $M$.  Exactly one neighbor $v$ of $u$ has a smaller label than $u$, and for this neighbor $v$, it can be shown that $\left| K_v \left( T_n^i \setminus \{u\} \right) \right|\geq n/M^3$ with good probability. 
For any other neighbor $v$ of $u$, $K_v \left( T_n^i \setminus \{u\} \right)$ is equal to the \emph{fringe subtree} of $v$ in $T_n^i$, which consists of the subtree of $T_n^i$ induced by all vertices whose path to $0$ contains $v$ (this includes $v$); see~Definition~\ref{def:frst}. Writing $F_{T_n^i}(v)$ for the size of the fringe subtree of $v$ in $T_n^i$, it is then possible to  judiciously choose $\ell,j$ so that with good probability,
\be{
\widehat{U}_{\ell,j}^{T_n^i}(u)\approx U_{\ell,j}^{T_n^i}(u):=\sum_{v=M+1}^{n} \mathbbm{1} \left\{ v\sim_{i} u , F_{T_n^i} (v) \in \left( \frac{n}{e^j} , \frac{n}{e^\ell} \right] \right\}.
}
The fringe subtree of vertex $v\in\{0,\ldots,n\}$ evolves as the number of red balls in a standard P\'olya urn started with $1$ red ball and $v$ blue balls, so
up to first order in large $n$ and $v$, we have that 
\be{
F_{T_n^i}(v) \approx n/v,
}
and so the size of the fringe subtree roughly identifies $v$. Using this identification,   we can  show that for appropriate $\ell,j$, 
\be{
U_{\ell,j}^{T_n^i}(u) \approx \sum_{v=e^\ell}^{e^j - 1} \mathbbm{1} \left\{ v\sim_{i} u \right\}.
}
Putting this all together, for any fixed 
$u_1, u_2\in \{0,\ldots,M\}$,
we have that 
\ba{
\left( \widehat{U}_{jk^2,(j+1)k^2}^{T_n^1}(u_1) - \widehat{U}_{jk^2,(j+1)k^2}^{T_n^2}(u_2) \right)^2 
    &\approx \left({U}_{jk^2,(j+1)k^2}^{T_n^1}(u_1) - {U}_{jk^2,(j+1)k^2}^{T_n^2}(u_2) \right)^2 \\
    &\approx\left( \sum_{v=e^{jk^2}}^{e^{(j+1)k^2}-1} \left(\mathbbm{1} \left\{ v\sim_{1} u_1 \right\}-\mathbbm{1} \left\{ v\sim_{2} u_2 \right\}\right)\right)^2,
}
where the summands on the right-hand side are independent, mean zero with 
\be{
\Var\left(\mathbbm{1}\left\{ v\sim_{1} u_1 \right\}-\mathbbm{1} \left\{ v\sim_{2} u_2 \right\}\right)
\sim
    \frac{2(1-\alpha\mathbbm{1}\{u_1=u_2\})}{v}.
}
Thus the mean of our statistic identifies $2(1-\alpha)$ as long as $u_1=u_2$. Furthermore, we show that for $k$ large, the  statistic is close in probability to its mean; see Proposition~\ref{prop:moment variance X}, which is the source of our choice of scaling of order $e^{j k^2}$ in the definition of $\widehat U_{jk^2, (j+1)k^2}$. 
Noticing that the mean is smaller when the vertices are the same, we should then take the minimum over a fixed set of vertices of each tree that have at least one vertex in common. As with our previous proof, this is accomplished with the sets $A_K^1$ and $A_{K}^{2}$, which, with good probability, can be thought of as fixed after some large time $M$, using \cite[Theorem~2]{Jog2017}, and have non-empty intersection, by Corollary~\ref{cor:intersection}.

\subsection{Open problems}\label{sec:openprobs}
We conclude the introduction by collecting open problems and future directions of research that stem from this work. 

\begin{itemize}
\item \textbf{(Rates of convergence)} 
What is the rate of convergence of $\widehat \alpha_n$ to $\alpha$ in Theorem~\ref{thm:estimation}? More generally, if we have a sequence of parameters $\alpha_{n} \to \alpha$ as $n\to\infty$, how quickly can $\abs{\alpha_n- \alpha}\to 0$ with also $\TV(\alpha_n{\text -} \CUA(n), \alpha{\text -}\CUA(n))\to 1$? The best possible rate of convergence that would come from our approach is $1/\mathrm{polylog}(n)$, which is likely not optimal.

\item \textbf{(Correlated centrality sets)} 
As previously mentioned, our first approach to  separating the models with $\alpha\not=\alpha'$ in general fails because our probabilistic bounds on  $\left| A_{K}^{1} \cap A_{K}^{2} \right|/K$ are independent of $\alpha$. 
Can this quantity be more precisely understood? 
Is there a limit 
$\lim_{K \to \infty} \left| A_{K}^{1} \cap A_{K}^{2} \right| / K$?
These questions are interesting in their own right in the context of understanding centrality, and could provide alternative proofs to some of our results.

\item \textbf{(Locations of central vertices)} 
Theorem~\ref{thm:centrality} shows that for $c > 1$ and  appropriately large $K$ and $n$,
\be{
\frac{| A_{\ceil{cK}}(n) \cap \{0,\ldots,K-1\}|}{K}\geq c(1-e^{-1/c})- \delta
}
with probability close to one. Is the lower bound optimal? What else can be said about the ratio? A tantalizing conjecture is that it converges to $c(1-e^{-1/c})$ as $K\to\infty$.

\item \textbf{(Other statistics)} 
Our statistics for detection and estimation all involve a two step procedure of 
(1) fixing a set of vertices based on centrality and 
(2) computing local statistics related to those vertices, 
and then, importantly, doing this across multiple time scales. 
It is natural to ask if there are simpler statistics (e.g., ones that are not multi-scale in nature) that can detect correlation or consistently estimate the correlation parameter. 

Our first approach to the detection problem was to build a statistic from the joint empirical degree distributions. Marginally, these evolve according to a generalized P\'olya urn, and, as shown by~\cite{Janson2005}, their centered and scaled limits are jointly Gaussian. 
The joint empirical degree distributions of correlated UA trees evolve according to a stochastic recursion, and applying central limit theorems in that setting (e.g., \cite{Fabian1968} and \cite{Hernandez2019}) yields that the centered and scaled joint degree counts are asymptotically jointly Gaussian with a covariance depending on $\alpha$. While we were able to
follow this program to show that 
$\TV \left( \alphaCUA \left( n \right), \UA(n)^{\otimes 2} \right)$
is asymptotically bounded away from $0$ (specifically, by counting leaves in the two trees), we were not able to show that it converges to $1$. See Appendix~\ref{sec:leaves} for more details.

Another approach to the detection problem is to consider the size of the largest common subtree of the two trees. However, even studying this statistic for two independent UA trees is interesting and challenging; see the forthcoming work by 
\citet*{Baumler2026}. 

\item \textbf{(Properties of $\alphaCUA$)} The UA model (also called the random recursive tree) is a beautiful mathematical object with a long history and connections to other areas, such as extreme value theory, with many fine properties understood and studied to this day (e.g., 
\citet*{Bjorklund2025},
\citet*{Eslava2022}, 
\citet*{Eslava2023}, 
\citet*{Janson2005, Janson2019},
\citet*{Lodewijks2024a},  
\citet*{Meir1978}, 
\citet*{Pittel1994}). Deriving analogs of single-tree properties for the $\alphaCUA$ model is an interesting avenue of future~research.

\item \textbf{(Other models of correlated networks)}  
Studying the questions of this paper in different models with a similar flavor is also of interest. A concrete model is the weighted random recursive tree (see, e.g., \cite{Borovkov2006}). Here vertex $i$ arrives with a (random or deterministic) weight $w_i$, and vertex $n>i$ attaches to vertex $i$ with probability $w_i/\sum_{j=0}^{n-1} w_j$. The case where the weights are all equal is the UA tree. There are several ways to introduce correlation into this model. First, the joint weights $\{(w_{i}^{1}, w_{i}^{2})\}_{i \geq 0}$ could be i.i.d.\ random but $w_{i}^{1}$ and $w_{i}^{2}$ could be correlated in some way, and the edge-choice mechanism then conditionally independent given the weights. Another option---analogous to the model studied in this paper---is to have fixed weights, but the edge choice mechanisms are aligned with probability $\alpha\in[0,1]$, and independent if not aligned. 

Another natural extension of correlated UA trees is to consider correlated UA \emph{graphs}, where now each arriving vertex has $m>1$ edges connecting independently and uniformly to existing vertices; here $m$ may be deterministic or random. 
We believe that our high-level approach to the detection problem should extend to this setting, with some modifications. 
While the notion of centrality we use to identify early vertices is tree-specific, 
we conjecture that 
replacing it with some other statistic that identifies early vertices and studying analogous subsequent local statistics such as degrees 
should yield similar results. 
For example, \citet*{Briend2023} use ``double cycles'' instead of centrality to identify early vertices in the setting of UA graphs. 
On the other hand, the local statistic of our consistent estimator of $\alpha$ is tree-specific, and it is not clear how to extend it to graphs. 
\end{itemize}

\subsection{Outline}\label{sec:outline}

The outline of the rest of the paper is as follows. 
In Section~\ref{sec:intcent}, we prove  Theorem~\ref{thm:centrality} and Corollary~\ref{cor:intersection}. We then show in  Section~\ref{sec:distind} that the separating statistic in~\eqref{eq:sepstat} can distinguish an independent pair of trees from a correlated pair.
Section~\ref{sec:genres} proves Theorem~\ref{thm:estimation} by showing that the estimator in~\eqref{eq:consest} converges in probability to $\alpha$. 
Some results about P\'olya urns that are used to control fringe-subtree sizes 
are deferred to Section~\ref{sec:polbd}.

\section{Finding early vertices}\label{sec:intcent}
The goal of this section is to prove Corollary~\ref{cor:intersection} via the following result, which both implies, and is a more detailed version  of Theorem~\ref{thm:centrality}  from Section~\ref{sec:earlyvert}. 
\begin{theorem}\label{thm:centrality-detailed}
Let $T_{n} \sim \UA(n)$ and recall that $A_K:=A_{K}(n)$ denotes the set of the $K$ most central vertices in~$T_{n}$. 
There exists a universal constant $C < \infty$ 
such that 
for every $K\in \N$, $c > 1$, and $\delta \in(0,c-1)$,
the following holds for all large enough $n \geq n_{0}(K,\delta,c)$: 
\[
\p \left( \babs{ A_{\ceil{c K}} (n) \cap \bclc{0,\ldots, K-1 }
}\geq \bclr{c (1-e^{-1/c})-\delta}K \right) 
\geq 1 - C \left(\frac{   c \log(c K) }{\delta ^2 K}+(0.7)^{\delta K/2}\right).
\]
\end{theorem}

As previously mentioned, 
the strategy to prove Theorem~\ref{thm:centrality-detailed} is
to relate centrality to certain subtree sizes. For this, we introduce a kind of inverse or negative centrality of a vertex $v$ in a tree $T$ on $\abs{\tree}$ vertices, which is simply
\be{
\Psi_\tree(v):=\abs{\tree}-\psi_\tree(v).
}
In words, $\Psi_\tree(v)$ is the number of vertices left in $\tree$ when removing the largest component subtree formed by removing vertex~$v$ and all edges connecting to it.  
The reason for looking at this quantity over $\psi_\tree(v)$ is that in a rooted tree $\tree$, if the largest subtree is the one containing the root, then $\Psi_\tree(v)$ is exactly the size of the so-called \emph{fringe subtree}, which we now define. 
\begin{definition}[Fringe subtree size]\label{def:frst}
The \emph{fringe subtree} of a vertex $v$ of $\tree$ having root $0$ is the subtree of $T$ induced  by the collection of vertices whose path to the root $0$ contains $v$ (note that this set includes $v$). We denote the number of vertices in the fringe subtree of $T$ by $F_\tree(v)$.
\end{definition}

With this heuristic in mind, we introduce the following (deterministic) definition of trees 
where vertices having ``large'' fringe trees appear early on.

\newcommand{\aone}{ \delta K/2}
\newcommand{\atwo}{n/(c_1 K)}
\newcommand{\afou}{\ceil{(c_1+\delta)K} -\aone}

\begin{definition}\label{def:delgd}
    For $K\in \N$, $\delta\in (0,1)$, and $c_1\geq 1$,  we say a tree $\tree$ on vertices $\{0,\ldots, n\}$  is $(K,\delta, c_1)$-good if
    \begin{enumerate}[label=(\roman*)]
      \item\label{item:bigfbigl} $\sum_{\aone-1< k \leq n} \mathbbm{1}_{\left\{ F_{\tree}(k) \geq 0.4 n \right\} } = 0$,
        \item\label{item:bigfanyl} $\sum_{k=0}^{n} \mathbbm{1}_{\left\{F_{\tree}(k) \geq \atwo \right\} } \leq \afou$, and
        \item\label{item:bigflitl} $\sum_{k=0}^{K-1} \mathbbm{1}_{\left\{ F_{\tree}(k) \geq \atwo \right\} } \geq  ( c_1 (1-e^{-1/c_1})- \delta/2)K$.
        \end{enumerate}
\end{definition}

The way to understand the definition is that if the fringe subtree of a vertex is big enough, then it will be central. Then $\ref{item:bigfanyl}$ caps the total number of vertices with large fringe subtrees, and~$\ref{item:bigflitl}$ implies that many vertices early in the tree will have such large fringe subtrees. (The first point controls differences between $\Psi_T$ and $F_T$.)
In particular, the definition allows us to quantify
where  the most central vertices are located.

\begin{lemma}\label{lem:centear}
If $\tree$ is a $(K,\delta,c_1)$-good tree on $\{0,\ldots, n\}$ with centrality sets denoted by $A_r$, $r=1,\ldots, n$, then
\be{
\babs{ A_{\ceil{(c_1+\delta) K}} \cap \bclc{0,\ldots, K-1 }
}\geq \bclr{c_1 (1-e^{-1/c_1})-\delta}K.
}
\end{lemma}

To apply Lemma~\ref{lem:centear} to a $\UA$ tree, we lower bound the probability that a $\UA$ tree is $(K,\delta,c_1)$-good. For this, we note that for a $\UA$ tree $\tree_n$, the quantity $(F_{\tree_n}(j))_{n\geq j}$ evolves as the number of blue balls in a classical P\'olya urn which starts (at step $n=j$) with 1 blue ball and $j$ white balls. (A standard P\'olya urn means that at each step, a ball from the urn is chosen uniformly at random and returned along with another of the same color.) Thus we can estimate the mean of the quantities appearing in Definition~\ref{def:delgd} by estimating probabilities of such P\'olya urns, which we do using de Finetti/beta-binomial representations and couplings. We can further estimate variances by using that for vertices $j<l$, $(F_{T_n}(j),F_{T_n}(l))_{n\geq l}$ evolves like a multicolor P\'olya urn started (at step $n=l$) with $l+1$ total balls.
See Section~\ref{sec:uadelgd} for details.

\begin{lemma}\label{lem:uadelgd}
Let $T_n$ be a vertex-labeled $\UA$ tree and fix parameters $K\in \N$, $\delta\in (0,1)$, and $c_1\geq 1$. Then there exists a universal constant $C^\star < \infty$ such that for all large enough $n \geq n(K,\delta,c_1)$, we have
    \begin{equation*}
        \p \left(  T_n  \text{ is $(K,\delta,c_1)$-good} \right) \geq 1- C^\star\left(\frac{   c_1 \log(c_1 K) }{\delta ^2 K}+(0.7)^{\delta K}\right).
    \end{equation*}
\end{lemma}

Given these lemmas, we can prove Theorem~\ref{thm:centrality-detailed}.
\begin{proof}[Proof of Theorem~\ref{thm:centrality-detailed}]
Applying Lemmas~\ref{lem:centear} and~\ref{lem:uadelgd} with $c_1=c-\delta>1$ implies
\besn{\label{eq:ack1}
\p \bbbclr{ \babs{ A_{\ceil{c K}} (n) \cap \bclc{0,\ldots, K-1 }
}\geq \bclr{(c-\delta)& (1-e^{-\frac{1}{c-\delta}})-\delta}K} \\
    &\geq 1 - C^\star \left(\frac{   (c-\delta) \log((c-\delta) K) }{\delta ^2 K}+(0.7)^{\delta K}\right).
}
Simple considerations imply that the function $\phi(x):= x(1-e^{-1/x})$ restricted to $[0,\infty)$ is non-decreasing with derivative $0\leq \phi'(x)\leq 1$, so that the mean value theorem implies
\be{
(c-\delta) \bclr{1-e^{-\frac{1}{c-\delta}}} - c (1-e^{-1/c})) \geq  - \delta,
}
and so we have that
\besn{\label{eq:ack2}
\p \bbbclr{ \babs{ A_{\ceil{c K}} (n) \cap &\bclc{0,\ldots, K-1 }
}\geq \bclr{c (1-e^{-1/c})-2\delta}K} \\
    &\geq\p \bbbclr{ \babs{ A_{\ceil{c K}} (n) \cap \bclc{0,\ldots, K-1 }
}\geq \bclr{(c-\delta) (1-e^{-\frac{1}{c-\delta}})-\delta}K}.
}
The theorem follows by 
combining~\eqref{eq:ack1} and~\eqref{eq:ack2}, 
changing $\delta$ to $\delta/2$, 
and using that $x \log x$ is increasing for $x \geq 1$. 
\end{proof}

We finish the subsection by proving  Corollary~\ref{cor:intersection}.

\begin{proof}[Proof of Corollary \ref{cor:intersection}] 
If $\abs{A_K^1\cap A_K^2} \leq 0.2 K$, then it must be the case that for at least one $i\in\{1,2\}$,
\be{
\babs{A_K^i \cap \{0,\ldots, \ceil{K/1.1}\}}\leq  (0.58)K. 
}
This is because the minimum overlap of two subsets of sizes larger than $A$ from a common superset of size $B<2A$ is 
$(2A-B)$. Taking $A=(0.58)K$ and $B=\ceil{K/1.1}+1$, this gives the overlap of size at least
$(2\cdot 0.58-(1.1)^{-1})K- 1 \geq 0.25 K$ for large enough $K$. 

But for fixed $\wt K\in \N$, Theorem~\ref{thm:centrality-detailed} with $c=1.1$, $\delta=0.001$, and $K=\wt K/c$ implies that 
\be{
\p\bclr{\abs{A_{\wt K} \cap \{ 0,\ldots, \ceil{\wt K/1.1}\}}\leq (0.58) \wt K }\leq \wt C \frac{\log(\wt K)}{\wt K},
}
and now switching $\wt K$ to $K$ and taking a union bound implies the result.
\end{proof}

In the remainder of this section, we prove Lemmas~\ref{lem:centear} and~\ref{lem:uadelgd}.

\subsection{Proof of Lemma~\ref{lem:centear}}

\begin{proof}[Proof of Lemma~\ref{lem:centear}]
First note that $A_r$ consists of the $r$ vertices~$j$ with the largest values of $\Psi_\tree(j)$. Thus, we relate $\Psi_\tree(k)$ to $F_\tree(k)$. A first key observation is that if $F_\tree(k) < 0.5 (n+1)$, then $F_\tree(k) = \Psi_\tree(k)$. Thus~$\ref{item:bigfbigl}$ from Definition~\ref{def:delgd} implies that 
\ben{\label{eq:pfdiff}
\sum_{k=0}^{n} \mathbbm{1}_{\clc{F_\tree(k)\not=\Psi_\tree(k)}}\leq \aone.
}
Using this with~$\ref{item:bigfanyl}$, we have that 
\be{
\sum_{k=0}^{n} \mathbbm{1}_{\clc{\Psi_\tree(k) \geq \atwo }} \leq \sum_{k=0}^{n} \mathbbm{1}_{\left\{F_{\tree}(k) \geq \atwo \right\} }+ \aone \leq\ceil{(c_1+\delta) K}.
}
 Thus there are at most $\ceil{(c_1+\delta) K}$ vertices $k$ with $\Psi_\tree(k) \geq \atwo$, which implies that any such vertex is in $A_{\ceil{(c_1+\delta)K}}$.
Using this fact,~\eqref{eq:pfdiff}, and~$\ref{item:bigflitl}$, we  find
that
\ba{
\babs{ A_{\ceil{(c_1+\delta) K}} \cap \bclc{0,\ldots,K-1 } }
    &\geq \sum_{k=0}^{K-1} \mathbbm{1}_{\clc{\Psi_\tree(k) \geq \atwo }}  
    \geq \sum_{k=0}^{K-1 } \mathbbm{1}_{\clc{F_\tree(k) \geq \atwo }} - \aone\\
    &\geq \bclr{c_1 (1-e^{-1/c_1})-\delta}K. \qedhere
}
\end{proof}

\subsection{Proof of Lemma~\ref{lem:uadelgd}}\label{sec:uadelgd}
To prove Lemma~\ref{lem:uadelgd}, we need estimates of the moments of the random variables appearing in  Definition~\ref{def:delgd}. We use properties of standard P\'olya urns to derive the following.

\begin{lemma}\label{lem:ftpbds}
Let $\tree_n\sim\UA(n)$ and fix constants  $0<s<\beta\leq 1$. Then for all $k=0,1,\ldots, n$, we have
\ben{\label{eq:ftub}
\P(F_{T_n}(k) \geq \beta n)  \leq  (1-\beta +s)^k + e^{-n s^2/2}.
}
For any $0<\gamma<1$, $0<\beta<1$, and $k\leq \gamma n$, with $ \frac{\beta  + 2\gamma + 2/n}{1-\gamma} \leq 1$ we have 
\ben{\label{eq:ftlb}
\P(F_{T_n}(k) \geq \beta n) \geq  \bbbclr{1- \frac{\beta  + 2\gamma + 2/n}{1-\gamma}}^k.
}
\end{lemma}
\begin{proof}
Recall that $F_{T_n}(k)$ is distributed as the  number of red balls in a standard P\'olya urn  started with $1$ red ball and $k$ blue balls after $n-k$ draws and replacements. We use the de Finetti representation $F_{T_n}(k)-1 \sim \mathrm{Binomial}(n-k, \eta_k)$, where $\eta_k\sim \mathrm{Beta}(1,k)$; see \cite[Theorems~2.1 and~2.2]{Freedman1965}. Conditioning on the value of $\eta_k$ and using stochastic monotonicity of the binomial distribution implies that for any threshold value~$t$, 
\ba{
\P(F_{T_n}(k) \geq \beta n) 
    \leq \P(F_{T_n}(k) \geq \beta n | \eta_k \leq t ) + \P(\eta_k \geq t ) 
    \leq \P(B_{n,t} \geq \beta n) + (1-t)^k,
}
where $B_{n,t}\sim\mathrm{Binomial}(n, t)$. Azuma's inequality implies that for $0<t<\beta$ we have
\be{
\P(B_{n,t} \geq \beta n)\leq \exp\bbclc{-\frac{(\beta n- t n )^2}{2n}}.
}
Choosing $t=\beta-s$ gives the first inequality~\eqref{eq:ftub}.

 For~\eqref{eq:ftlb},
Lemma~\ref{lem:purntbd} below with $r=1$, $b=k$ and $m=n-k$ implies that
\ba{
\P(F_{T_n}(k) \geq \beta n)
    & \geq \P\left(\eta_k\geq \frac{\beta n + 2(k+1)}{n-k} \right) 
    \geq  \P\left(\eta_k\geq \frac{\beta  + 2\gamma + 2/n}{1-\gamma} \right)\\ 
    &= \bbbclr{1- \frac{\beta  + 2\gamma + 2/n}{1-\gamma}}^k,
}
where the second inequality is because of monotonicity in $k$. 
\end{proof}

We use Lemma~\ref{lem:ftpbds} to bound moments of the quantities related to the $(K,\delta,c_1)$-good Definition~\ref{def:delgd}. 

\begin{lemma}\label{lem:delgdmom}
Let $\tree_n\sim \UA(n)$. For any $\delta\in (0, 0.01)$,  
$c_1>1$, and $K < n^{1/4}/c_1$, there is a universal constant $C$, and an $n(K,\delta,c_1)$ depending only on $(K,\delta,c_1)$ such that for all $n\geq n(K,\delta,c_1)$, we have
\ban{
\sum_{\aone-1< k\leq n} \P(F_{\tree_n}(k) \geq 0.4 n) &\leq C (0.7)^{\delta K}, \label{eq:dgp1}\\
\sum_{k=0}^n\P(F_{\tree_n}(k) \geq  \atwo) &\leq \bclr{1 + \tfrac{\delta}{4 c_1} }c_1 K, \label{eq:dgp2}\\
\sum_{k=0}^{K-1} \P(F_{\tree_n}(k) \geq  \atwo)& \geq ( c_1 (1-e^{-1/c_1})- \delta/4)K, \label{eq:dgp3}
}
and for any $L\in \{0,1,2,\ldots,n\}$, we have
\ben{\label{eq:dgvar}
\Var\bbbclr{ \sum_{k=0}^L \mathbbm{1}_{\{ F_{\tree_n}(k) \geq \atwo \}}}  \leq C\bclr{ c_1 K\log(c_1 K)}.
}
\end{lemma}
\begin{proof}
For~\eqref{eq:dgp1}, we use~\eqref{eq:ftub} with $\beta=0.4$ and $s=0.1$ to find
\ba{
\sum_{k=\aone-1}^n \P(F_{\tree_n}(k) \geq 0.4 n)
    &\leq \sum_{k=\aone-1}^n \bclr{(0.7)^k + e^{-0.005 n}} \\
    &\leq 0.7^{\aone-1} \frac{1}{0.3} + n e^{-0.005 n} \\
    &\leq C \, (0.7)^{\delta K}
}
for some constant $C$, and using that $n$ is large enough. 

For~\eqref{eq:dgp2}, we use~\eqref{eq:ftub} with $\beta=1/(c_1 K)$ and $s=n^{-1/4}$ (noting that $\beta>s$ by assumption) to find
\ba{
\sum_{k=0}^n\P(F_{\tree_n}(k) \geq  n/(c_1 K)) 
    &\leq \sum_{k=0}^n \bclr{(1-(c_1 K)^{-1} +n^{-1/4})^k + e^{-\sqrt{n}/2}}  \\
    &\leq \frac{1}{(c_1K)^{-1} - n^{-1/4}} + (n+1) e^{-\sqrt{n}/2}  \\
    &\leq \bclr{1 + \tfrac{\delta}{4 c_1}} c_1 K,
}
for $n$ large enough. 

For~\eqref{eq:dgp3}, we use~\eqref{eq:ftlb}. Set $\gamma$ small enough and $n$  large enough (depending on $\delta$) so that
\ben{\label{eq:dgmn}
1- \frac{(c_1 K)^{-1}  + 2\gamma + 2/n}{1-\gamma} \geq 1- \frac{1+\iota}{c_1K},
}
where $\iota$ is a parameter that can be chosen sufficiently small later. With this choice of $\gamma$ and possibly enlarging $n$ so that  $K < \gamma n$,~\eqref{eq:ftlb} with $\beta=1/(c_1 K)$ and~\eqref{eq:dgmn} implies that 
\ba{
\sum_{k=0}^{K-1} \P(F_{\tree_n}(k) \geq  n/(c_1K))
    & \geq \sum_{k=0}^{K-1} \left(1- \frac{1+\iota}{c_1 K}\right)^k
    = \frac{ 1- \left(1-\frac{1+\iota}{c_1 K}\right)^{K}}{1-\left( 1-\frac{1+\iota}{c_1 K}\right) } \\
    & \geq  c_1 K \left( \frac{1-e^{- \frac{(1+\iota)}{c_1}}}{1+\iota}\right) 
    \geq  ( c_1 (1-e^{-1/c_1})- \delta/4)K,
}
which follows by shrinking $\iota$, and then choosing $n$ large enough.

Next we bound the variance~\eqref{eq:dgvar}, and to save space  we denote  $I_k:=\mathbbm{1}_{\{ F_{\tree_n}(k) \geq n/(c_1K) \}}$. 
We compute
\ben{\label{eq:var2bd}
\Var\bbbclr{ \sum_{k=0}^L I_k} = \Var\bbbclr{ \sum_{k=1}^L I_k}
    =\sum_{k=1}^L \Var(I_k) + 2\sum_{1\leq k < l \leq L}\Cov(I_k,I_l).
}
For the first term of~\eqref{eq:var2bd}, we have the bound
\ben{\label{eq:varbdt1}
\sum_{k=1}^L \Var(I_k)\leq \sum_{k=1}^L \P(F_{\tree_n}(k) \geq n/(c_1 K) ) \leq 2 c_1 K,
}
where we used~\eqref{eq:dgp2} in the last inequality.

Considering the covariance terms $\Cov(I_k,I_l)$ for $k<l$, 
let $E_{k,l}$ denote the event that vertex~$l$ connects to the fringe subtree of $k$ in $\tree_l$ (and is hence in the fringe subtree for all time). Two key facts to note are that 1) the event $E_{k,l}$ is independent of $\{F_{\tree_n}(l) \geq n/(c_1 K)\}$ (since the two events are functions of disjoint collections of the edge variables of $\tree_n$), and 2) on the event $E_{k,l}^c$, the process $(F_{\tree_m}(k),F_{\tree_m}(l))_{m\geq l}$ evolves like a three color standard P\'olya urn, and so $F_{\tree_n}(k)$ and $F_{\tree_n}(l)$ are negatively related, which implies the conditional covariance of $I_k$ and $I_l$ is non-positive. Thus we have
\ba{
\Cov(I_k, I_l) & = \E\bcls{ I_l (I_k - \E[I_k])}\\
    & = \E\bcls{ \mathbbm{1}_{E_{k,l}} (I_k - \E[I_k])|I_l=1}\E[I_l]+ \E\bcls{ \mathbbm{1}_{E_{k,l}^c} I_l (I_k - \E[I_k])}\\
    &\leq \P\bclr{E_{k,l}|I_l=1}\E[I_l]+ \E\bcls{I_l(I_k - \E[I_k])|E_{k,l}^c}\P(E_{k,l})\\
    &\leq \P(E_{k,l})\E[I_l]\\
    & \leq \frac{1}{k} \bclr{\clr{1- (c_1 K)^{-1} + n^{-1/4}}^l + e^{-\sqrt{n}/2}},
}
where the second inequality is because of the independence and negative conditional covariance already mentioned, and the third uses~\eqref{eq:ftub} with $\beta=1/(c_1K)$ and $s=n^{-1/4}$. Thus we find that the covariance term of~\eqref{eq:var2bd} is, for $n$ large enough, upper bounded by
\ba{
\sum_{1\leq k < l \leq L}\Cov(I_k,I_l) &  \leq \sum_{1\leq k < l \leq L} \frac{1}{k} \bclr{\clr{1- (c_1 K)^{-1}+ n^{-1/4}}^l + e^{-\sqrt{n}/2}} \\
&
\leq \sum_{k=1}^L \frac{1}{k} \sum_{l=k+1}^{L} \bclr{1-(2 c_1 K)^{-1}}^l + L^2 e^{-\sqrt{n}/2}\\
&
\leq \sum_{k=1}^L \frac{1}{k} (2c_1 K) \bclr{1-(2 c_1 K)^{-1}}^k + n^2 e^{-\sqrt{n}/2}
\\
&
\leq C \bclr{c_1 K  \log(c_1 K)},
}
which together with~\eqref{eq:varbdt1} implies~\eqref{eq:dgvar}.
\end{proof}

We can use the moment information in Lemma~\ref{lem:delgdmom} to prove Lemma~\ref{lem:uadelgd}.

\begin{proof}[Proof of Lemma~\ref{lem:uadelgd}]
We show that the probability of each of the complements of the three events in Definition~\ref{def:delgd} is of order at most $\log( K)/K$, and the result then follows from a union bound.
For~$\ref{item:bigfbigl}$, Markov's inequality and~\eqref{eq:dgp1} imply 
\be{
\P\bbbclr{\sum_{\aone< k \leq n} \mathbbm{1}_{\left\{ F_{\tree_n}(k) \geq 0.4 n \right\} } \geq 1 } \leq \sum_{\aone< k \leq n} \P(F_{\tree_n}(k) \geq 0.4 n) \leq C (0.7)^{\delta K}.
}
For~$\ref{item:bigfanyl}$, the bounds~\eqref{eq:dgp2} and~\eqref{eq:dgvar} on the mean and variance combined with Chebyshev's inequality imply 
\ba{
\P\bbbclr{\sum_{k=0}^{n} \mathbbm{1}_{\left\{F_{\tree}(k) \geq \atwo \right\} } > \afou }	&\leq \frac{ C\bclr{ c_1 K\log(c_1 K)} }{\bclr{\afou-(1+\delta/(4c_1))c_1K}^2} \\
	&\leq \frac{16C  \bclr{c_1  \log(c_1 K)}}{\delta^2 K}.
}
Similarly, for~$\ref{item:bigflitl}$, the bounds~\eqref{eq:dgp3} and~\eqref{eq:dgvar} on the mean and variance combined with Chebyshev's inequality imply 
\ba{
\P\bbbclr{\sum_{k=0}^{K-1} \mathbbm{1}_{\left\{ F_{\tree}(k) \geq \atwo \right\} } <  ( c_1 (1-e^{-1/c_1})- \delta/2)K}
	&\leq \frac{ 16 C\bclr{ c_1 \log(c_1 K)} }{\delta ^2 K}. \qedhere
}

\end{proof}

\section{Distinguishing independent trees from correlated ones}\label{sec:distind}

In this section, we show that the total variation distance between a pair of independent $\UA$ trees and a pair of correlated ones tends to one, as the number of vertices tends to infinity. As previously described, if we knew the labels of $A_K^{i}(n)$ for $i=1,2$, and the overlap between them, then the mean and variance of the separating statistic could be computed. We first derive the appropriate results in this setting in Lemma~\ref{lem:covariances}, and then use them in
Section~\ref{sec:distpf} to show how to distinguish.

\subsection{Mean and variance of sums of degrees of sets}\label{sec:moms}
 Fix $\ell \in \N$ and an increasing sequence  $0<K_0 <  K_1 < \cdots K_\ell\leq N$ of integers, and let $\clr{M_{K_j}}_{j=0}^\ell$ and $\clr{\wt M_{K_j}}_{j=0}^\ell$ be collections of subsets of $\{0,\ldots, N\}$ such that for all $j=0,1,\ldots, \ell$, 
\besn{\label{eq:gdcondvars}
    |M_{K_j}| = |\widetilde{M}_{K_j}| = K_j,  \text{ and }
    |M_{K_j} \cap \widetilde{M}_{K_j}| \geq 0.2 K_j.
}
Let $X_{N+1}, X_{N+2}, X_{N+3},\ldots$, and $Z_{N+1}, Z_{N+2}, Z_{N+3}, \ldots$, be independent random variables such that $X_i$ and $Z_i$ are uniformly distributed on $\{0,\ldots,i-1\}$. Further let $(U_i)_{i\in \N}$ be i.i.d.\ random variables that have a Bernoulli$(\alpha)$-distribution and are independent of $X_{N+1}, X_{N+2}, X_{N+3},\ldots$, and $Z_{N+1}, Z_{N+2}, Z_{N+3}, \ldots$. Define $Y_{N+1},Y_{N+2}, Y_{N+3}, \ldots$ by
\begin{equation*}
    Y_i = \begin{cases}
        Z_i & \text{ if } U_i=0,\\
        X_i & \text{ if } U_i=1.
    \end{cases}  
\end{equation*}
Note that the random variable $Y_i$ is a mixture of $X_i$ and $Z_i$. For $n\geq N+1$, define the random variables 
\besn{\label{eq:sjxyz}
    S_{j}^{X}(n) = \sum_{a=N+1}^{n} \mathbbm{1}_{\{X_a \in M_{K_j}\}}, \ \
   S_{j}^{Y}(n) = \sum_{a=N+1}^{n} \mathbbm{1}_{\{Y_a \in \wt M_{K_j}\}},  \ \ 
   S_{j}^{Z}(n) = \sum_{a=N+1}^{n} \mathbbm{1}_{\{Z_a \in \wt M_{K_j}\}}.
}

The main result of this section is the following.

\begin{lemma}\label{lem:covariances}
Recall the random variables $S_j^X(n), S_j^Y(n)$, and $S_j^Z(n)$, $j=0,1,\ldots,\ell$ just defined, and assume they satisfy~\eqref{eq:gdcondvars}. Then  there is an $m\in \N$ depending only on  $\ell, N, (K_j)_{j=0}^\ell$, and $\alpha$, such that for 
 all $0\leq j\leq i\leq \ell$ and $n>m$, we have 
	\begin{align}
		& \label{eq:variance lem 1}
		\E \left[ \left( S_{j}^{X}(n) - S_{j}^{Z}(n) \right)^2 \right] \geq \left(2 - 0.01 \alpha \right) K_j \log(n),
		\\
		& \label{eq:variance lem 2}
		\E \left[ \left( S_{j}^{X}(n) - S_{j}^{Y}(n) \right)^2 \right] \leq \left(2 - 0.1 \alpha \right) K_j \log(n),
		\\
		&
		\label{eq:variance lem 3}
		\Cov \left( \left( S_{j}^{X}(n) - S_{j}^{Z}(n) \right)^2, \left( S_{i}^{X}(n) - S_{i}^{Z}(n) \right)^2 \right) \leq 10 K_j^2 \log(n)^2 , \text{ and }
		\\
		&
		\label{eq:variance lem 4}
		\Cov \left( \left( S_{j}^{X}(n) - S_{j}^{Y}(n) \right)^2, \left( S_{i}^{X}(n) - S_{i}^{Y}(n) \right)^2 \right) \leq 10 K_j^2 \log(n)^2.
	\end{align}
Furthermore,
\ban{
\P\bbbclr{\frac{1}{\ell+1}\sum_{j=0}^{\ell}\frac{\bclr{ S_{j}^{X}(n) - S_{j}^{Z}(n)}^2}{K_j \log(n)} \leq 2-0.02\alpha }&\leq \frac{20\sum_{0\leq j \leq i \leq \ell}(K_j/K_i) }{(\ell+1)^2(0.01 \alpha)^2}, \label{eq:sepstatprbind} \\
\P\bbbclr{\frac{1}{\ell+1}\sum_{j=0}^{\ell}\frac{\bclr{ S_{j}^{X}(n) - S_{j}^{Y}(n)}^2}{K_j \log(n)}  \geq 2-0.09\alpha }&\leq \frac{20\sum_{0\leq j \leq i \leq \ell}(K_j/K_i) }{(\ell+1)^2(0.01 \alpha)^2}, \label{eq:sepstatprbcor}
}
and for any $t>0$,
\ban{
\P\bclr{\babs{ S_{j}^{X}(n) - S_{j}^{Z}(n)} \geq t}&\leq \frac{4 \log(n)}{t^2} K_j,  \label{eq:absdiffind}\\
\P\bclr{\babs{ S_{j}^{X}(n) - S_{j}^{Y}(n)} \geq t}&\leq \frac{4 \log(n)}{t^2} K_j. \label{eq:absdiffcor}
}
\end{lemma}
The left hand side of the lower bound~\eqref{eq:variance lem 1} does not depend on $\alpha$, and the right hand side could replace  $0.01\alpha$ with any small enough constant, but we prefer to state things this way for easier comparison to the upper bound~\eqref{eq:variance lem 2}.

Before proving the lemma, we record a useful result that follows easily by direct computation. 
\begin{proposition}
    	If $X,Y,\widetilde{X},\widetilde{Y}$ are square-integrable and mean-zero random variables such that $(X,\widetilde{X})$ is independent of $(Y,\widetilde{Y})$, then
		\begin{equation}\label{cov fact 1}
		\Cov \bclr{ XY,\widetilde{X}\widetilde{Y} } = \Cov \bclr{ X,\widetilde{X} } \Cov\bclr{ Y,\widetilde{Y}}.
	\end{equation}
	If, in addition, $Z$ is a square integrable random variable  such that $Y$ is independent of $(X,Z)$, then
	\begin{equation}\label{cov fact 2}
		\Cov(X,YZ)=0 .
	\end{equation}
\end{proposition}


\begin{proof}[Proof of Lemma \ref{lem:covariances}]
We start with computing $\E \bbcls{ \left( S_{j}^{X}(n) - S_{j}^{Z}(n) \right)^2 }$. Since the random variables $(X_a)_{a \geq N+1}$ are independent, we have
\begin{align*}
    \E \left[S_{j}^{X}(n)\right] &= \sum_{a=N+1}^{n} \p \left( X_a \in M_{K_j} \right)
    =
    \sum_{a=N+1}^{n} \frac{K_j}{a},
    \\
    \Var \left(S_{j}^{X}(n)\right) &= \sum_{a=N+1}^{n} \Var \left( \mathbbm{1}_{\{X_a \in M_{K_j} \} } \right)
    = \sum_{a=N+1}^{n} \left(\frac{K_j}{a} - \left(\frac{K_j}{a}\right)^2\right) \approx K_j \log(n),
\end{align*}
where we have used from~\eqref{eq:gdcondvars} that $|M_{K_j}|=K_j$.
Since $|M_{K_j}|=|\wt M_{K_j}|$, we see that $S_{j}^{X}(n), S_{j}^{Z}(n)$, and $ S_{j}^{Y}(n)$ are identically distributed, and in particular,  they all have the same mean and variance as above. Since $S_{j}^{X}(n)$ and $S_{j}^{Z}(n) $ are  independent, we easily find that
\ban{
     \E \bbcls{ \left( S_{j}^{X}(n) - S_{j}^{Z}(n) \right)^2 } &= 2 \Var \left( S_{j}^{X}(n) \right) = 2 \sum_{a=N+1}^{n} \left(\frac{K_j}{a} - \left(\frac{K_j}{a}\right)^2\right) \label{eq:meansqind} \\
    & \notag \geq 2 \sum_{a=N+1}^{n} \frac{K_j}{a} - 5 K_j^2 
    \geq \left(2 - 0.01 \alpha \right) K_j \log(n), \notag
}
for $n$ large enough, noting that we have introduced $0.01\alpha$ in line with the remark following the statement of Lemma~\ref{lem:covariances}. This shows inequality~\eqref{eq:variance lem 1}. 

For~\eqref{eq:variance lem 2}, the sets $ M_{K_j} $ and $ \widetilde{M}_{K_j}$ satisfy $| M_{K_j} | = | \widetilde{M}_{K_j}| =K_j$ and $| M_{K_j} \cap \widetilde{M}_{K_j}| \geq 0.2 K_j$ so that the symmetric difference of the sets $ M_{K_j} \Delta \widetilde{M}_{K_j} = \bclr{ M_{K_j} \cup \widetilde{M}_{K_j}} \setminus \bclr{ M_{K_j} \cap \widetilde{M}_{K_j}} $ satisfies $| M_{K_j} \Delta \widetilde{M}_{K_j}| \leq 1.8 K_j$. 
Using this, we can upper bound the second moment of $S_{j}^{X}(n) - S_{j}^{Y}(n)$ by
\begin{align*}
    & \E \left[ \left( S_{j}^{X}(n) - S_{j}^{Y}(n) \right)^2 \right] = \Var \left( \sum_{a=N+1}^n \left( \mathbbm{1}_{\{X_a \in M_{K_j} \} } - \mathbbm{1}_{\{Y_a \in \widetilde{M}_{K_j} \} } \right) \right) 
    \\
    &= \sum_{a=N+1}^n \Var   \left( \mathbbm{1}_{\{X_a \in M_{K_j} \} } - \mathbbm{1}_{\{Y_a \in \widetilde{M}_{K_j} \} } \right) 
    \\
    &
    =
    \alpha \sum_{a=N+1}^n \E \left[ \left(\mathbbm{1}_{\{X_a \in M_{K_j} \} } - \mathbbm{1}_{\{X_a \in \widetilde{M}_{K_j}  \} } \right)^2 \right]  + (1-\alpha) \sum_{a=N+1}^n \E \left[ \left(\mathbbm{1}_{\{X_a \in M_{K_j} \} } - \mathbbm{1}_{\{Z_a \in \widetilde{M}_{K_j}  \} } \right)^2 \right] \\
    &
    =
    \alpha \sum_{a=N+1}^n \E \left[ \left(\mathbbm{1}_{\{X_a \in M_{K_j} \Delta \widetilde{M}_{K_j}  \} } \right)^2 \right]  + (1-\alpha) \sum_{a=N+1}^n \Var \left(\mathbbm{1}_{\{X_a \in M_{K_j} \} } - \mathbbm{1}_{\{Z_a \in \widetilde{M}_{K_j}  \} } \right)
    \\
    &
    =
    \alpha  \sum_{a=N+1}^n \frac{\babs{ M_{K_j} \Delta \widetilde{M}_{K_j}}}{a}  + 2 (1-\alpha) \sum_{a=N+1}^n \left( \frac{K_j}{a} - \frac{K_j^2}{a^2} \right)
    \\
    &
    \leq
    \alpha 1.8 \sum_{a=N+1}^n \frac{K_j}{a}  + 2 (1-\alpha) \sum_{a=N+1}^n \frac{K_j}{a} 
    =
    \left(2 - 0.2 \alpha \right) \sum_{a=N+1}^n \frac{K_j}{a} \leq \left(2 - 0.1 \alpha \right) K_j \log(n),
\end{align*}
for $n$ large enough, which proves~\eqref{eq:variance lem 2}. 

Next, we estimate the covariances of $\bclr{\bclr{ S_{j}^{X}(n) - S_{j}^{Y}(n) }^2}_{j\in \{0,\ldots,\ell \} }$ to prove inequality~\eqref{eq:variance lem 4}. The proof does not depend on the fact that $\alpha > 0$, and thus the exact same proof (with $Y$ replaced by~$Z$) also shows inequality~\eqref{eq:variance lem 3}. 
First define the random variables $R_a:= \mathbbm{1}_{\{X_a \in M_{K_j} \} } - \mathbbm{1}_{\{Y_a \in \widetilde{M}_{K_j} \} }$, $a\geq N+1$, which are independent and  satisfy $R_{a} \in \{-1,0,1\}$, 
\ben{\label{eq:rmoms}
    \E\left[ R_a \right] = \E\left[ R_a^3 \right] = 0, \text{ and } \E\left[ R_a^2 \right] = \E\left[ R_a^4 \right] \leq 2 \p\left( X_a \in M_{K_j} \right) \leq \frac{2K_j}{a} .
    }
We define their index-$i$ counterpart random variables $T_a := \mathbbm{1}_{\{X_a \in M_{K_i} \} } - \mathbbm{1}_{\{Y_a \in \widetilde{M}_{K_i} \} }$ for $a=N+1,\ldots,n$, so that
\begin{equation*}
    \clr{S_{j}^{X}(n)-S_{j}^{Y}(n)} = \sum_{a=N+1}^{n} R_a , \ \text{ and }\ 
     \clr{S_{i}^{X}(n)-S_{i}^{Y}(n)} = \sum_{c=N+1}^{n} T_c.
\end{equation*}

Now, we calculate the covariance of $\bclr{ S_{j}^{X}(n) - S_{j}^{Y}(n) }^2$ and $\bclr{ S_{i}^{X}(n) - S_{i}^{Y}(n)}^2$ for $i,j \in \{0,\ldots,\ell\}$ with $j\leq i$. 
Noting that $(R_a,T_a)_{a\geq N+1}$ are independent mean-zero bivariate random vectors, we find that 
\ba{
\Cov \bbclr{\bclr{S_{j}^{X}(n)-S_{j}^{Y}(n)}^2, \bclr{ &S_{i}^{X}(n)-S_{i}^{Y}(n) }^2 } \\
    &= \sum_{a,b=N+1}^n  \sum_{\{c,d\}\cap \{a,b\}\not=\emptyset} \Cov( R_a R_b, T_c T_d) \\
    &=\sum_{a=N+1}^n \Cov(R_a^2, T_a^2) + 2 \sum_{ \substack{a,b\in \{N+1,\ldots,n\} : \\ a \neq b }} \Cov(R_a R_b, T_a T_b),
   }
   where the second equality follows because 
   $\Cov(R_a^2, T_a T_c)=0$ for $c\not=a$, and 
$\Cov(R_a R_b, T_a T_c)=0$ for $a\not=b$ and $c\not\in\{a,b\}$ (again by independence and zero means).
Using~\eqref{cov fact 1}, we have that for $a\not=b$, 
$\Cov(R_a R_b, T_a T_b) = \Cov(R_a, T_a) \Cov(R_b, T_b)$, and, using~\eqref{eq:rmoms} and $\abs{T_a}\leq 1$, that
\ba{
\abs{ \Cov(R_a, T_a)} &= \abs{\E[ R_a T_a]}\leq  \E[\abs{R_a}] \leq 2\frac{K_j}{a}, \\
\abs{\Cov(R_a^2, T_a^2)} &\leq \E[ R_a^2 T_a^2 ]\leq  \E[R_a^2] \leq 2\frac{K_j}{a}.
}
Combining the previous two displays gives
 \ba{
 \Cov \bbclr{&\bclr{S_{j}^{X}(n)-S_{j}^{Y}(n)}^2, \bclr{ S_{i}^{X}(n)-S_{i}^{Y}(n) }^2 } \\ 
 &\leq \sum_{a=N+1}^n \Cov(R_a^2, T_a^2) + 2 \sum_{ \substack{a,b\in \{N+1,\ldots,n\} : \\ a \neq b }} \Cov(R_a,T_a) \Cov( R_b,T_b) \\
 &\leq 2 \sum_{a=N+1}^n \frac{K_j}{a}+  8 \sum_{ \substack{a,b\in \{N+1,\ldots,n\} : \\ a \neq b }} \frac{K_j^2}{a b}  \\
 &\leq 10 K_j^2 \log(n)^2,
 }  
 for large enough $n$. This establishes~\eqref{eq:variance lem 4} for $j \leq i$.

To prove~\eqref{eq:sepstatprbind}, first note that inequality~\eqref{eq:variance lem 1} implies that
\begin{equation*}
    \E \left[ \frac{1}{\ell+1}\sum_{j=0}^{\ell}\frac{\bclr{ S_{j}^{X}(n) - S_{j}^{Z}(n)}^2}{K_j \log(n)} \right] \geq 2 - 0.01 \alpha
\end{equation*}
for all large enough $n \in \N$.  Chebyshev's inequality shows that
\ben{\label{eq:chebvarind}
\P\bbbclr{\frac{1}{\ell+1}\sum_{j=0}^{\ell}\frac{\bclr{ S_{j}^{X}(n) - S_{j}^{Z}(n)}^2}{K_j \log(n)} \leq 2-0.02\alpha } 
    \leq \frac{\Var\bbclr{\sum_{j=0}^{\ell}\frac{\bclr{ S_{j}^{X}(n) - S_{j}^{Z}(n)}^2}{K_j \log(n)}} }{(\ell+1)^2 (0.01 \alpha)^2}.
}
We further bound the variance using~\eqref{eq:variance lem 3} to find
\ba{
\Var\bbclr{\sum_{j=0}^{\ell}\frac{\bclr{ S_{j}^{X}(n) - S_{j}^{Z}(n)}^2}{K_j \log(n)}}&  \leq  2 \sum_{0\leq j \leq i \leq \ell}\frac{ \Cov \left( \left( S_{j}^{X}(n) - S_{j}^{Z}(n) \right)^2, \left( S_{i}^{X}(n) - S_{i}^{Z}(n) \right)^2 \right)}{K_i K_j \log(n)^2} \\
    &\leq 20 \sum_{0\leq j \leq i \leq \ell} (K_j/K_i),
}
which, combined with~\eqref{eq:chebvarind}, gives~\eqref{eq:sepstatprbind}. The proof of~\eqref{eq:sepstatprbcor} follows in a nearly identical manner to~\eqref{eq:sepstatprbind}, substituting~\eqref{eq:variance lem 2} for~\eqref{eq:variance lem 1}, and~\eqref{eq:variance lem 4} for~\eqref{eq:variance lem 3}.

To prove~\eqref{eq:absdiffind}, we use Markov's inequality and the squared moment expression~\eqref{eq:meansqind} to find that for $n$ large enough, 
\ba{
\P\bclr{\babs{ S_{j}^{X}(n) - S_{j}^{Z}(n)} \geq t} = \P\bclr{\bclr{ S_{j}^{X}(n) - S_{j}^{Z}(n)}^2 \geq t^2}\leq \frac{\E \bbcls{ \bclr{ S_{j}^{X}(n) - S_{j}^{Z}(n)}^2 }}{t^2} \leq \frac{4 \log(n)}{t^2} K_j.
}
The proof of~\eqref{eq:absdiffcor} follows in a nearly identical manner, but using~\eqref{eq:variance lem 2} in place of~\eqref{eq:meansqind}. 
 \end{proof}

\subsection{Using the moment bounds to distinguish}\label{sec:distpf}

The next two lemmas immediately imply that it is possible to asymptotically distinguish the independent $\UA$ model from a correlated one.

\begin{lemma}\label{lem:sep1}
    If $\bclr{ T_n\st{1}, T_n\st{2}}_{n\in \N}$ is distributed according to the measure $0\text{-CUA}$, then for all $\eps>0$ there exist $m,K,\ell \in \N$ such that for all $n \geq m$
    \ba{
        \p \left( \frac{1}{\ell + 1} \sum_{j=0}^{\ell} \frac{1}{2^j K \log(n)} \left( \sum_{x \in A_{2^j K}\st{1}(n)} \deg_{\tree_n\st{1}}(x) -  \sum_{x \in A_{2^j K}\st{2}(n)} \deg_{\tree_n\st{2}}(x) \right)^2 > 2-0.05 \alpha \right) > 1-\eps .
    }
\end{lemma}

\begin{lemma}\label{lem:sep2}
    If $\bclr{T_n\st{1}, T_n\st{2} }_{n\in \N}$ is distributed according to the measure $\alphaCUA$, then for all $\eps>0$ there exist $m,K,\ell \in \N$ such that for all $n \geq m$
    \ba{
        \p \left(\frac{1}{\ell + 1} \sum_{j=0}^{\ell} \frac{1}{2^j K \log(n)} \left( \sum_{x \in A_{2^j K}\st{1}(n)} \deg_{\tree_n\st{1}}(x) -  \sum_{x \in A_{2^j K}\st{2}(n)} \deg_{\tree_n\st{2}}(x) \right)^2  < 2-0.05 \alpha \right) > 1-\eps .
    }
\end{lemma}

The statistics in the lemmas are closely related to those considered in Lemma~\ref{lem:covariances} with $M_{K_j}=A_{K_j}\st1$ and $\wt M_{K_j}\st{2}$, but there are two key differences. 
The first is that the sets $A_{K_j}\st{i}$ are not fixed, and the second is that 
Lemma~\ref{lem:covariances} only considers the contribution to degrees past some threshold step $N$. 
The following two lemmas are a key tool for controlling these differences between the two setting, in order to apply Lemma~\ref{lem:covariances}. 
\begin{lemma}\label{lem:sepstatdetcon}
From two sequences $(X_j\st{i})_{j\geq 1}$, $i=1,2$ with $X_j\st{i} \in \{0,1,\ldots, j-1\}$, generate a paired tree sequence $(\tree_k\st{1}, \tree_k\st{2})_{k\geq 1}$, where the edges of $\tree_k\st{i}$ are $\{\{j, X_j\st{i}\}\}_{j=0}^k$.
    Fix a sequence of time steps $0<K_0<K_1<\cdots< K_\ell< m<n$, and assume  $(X_j\st{1},X_{j}\st{2})_{j\geq 1}$  are such that the following hold. 
    \begin{enumerate}[label=(\alph*)]
        \item\label{item:kcentralstat} For all $j=0,\ldots, \ell$, the $K_j$ most central nodes of $(\tree_k\st{i})_{k\geq 1}$, $i=1,2$, do not change after step $m$, meaning
        \[
        A_{K_j}\st{i}(N) = A_{K_j}\st{i}(m) \text{ for all } N \geq m, \, i=1,2, \, \text{and } j=0,\ldots, \ell.
        \]
        \item\label{item:kcentdegsmall} After time step $m$, the differences in the number of nodes connecting to $A_{K_j}\st{i}$ in each of the two trees has the following uniform bound
        \ba{
         \max_{j \in \{0,\ldots, \ell \}} \left| \sum_{k=m+1}^{n} \mathbbm{1}_{ \left\{ X_k\st{1} \in A_{K_j}\st{1}(m) \right\} }  - \sum_{k=m+1}^{n} \mathbbm{1}_{ \left\{ X_k\st{2} \in A_{K_j}\st{2}(m)\right\} } \right| \leq \log(n)^{2/3}.
        }
         \item\label{item:sepstatcrit} After time step $m$, the weighted sum of squares of the differences in the number of nodes connecting to $A_{K_j}\st{i}(m)$ in each of the two trees satisfies the bound
        \[
        \frac{1}{\ell + 1} \sum_{j=0}^{\ell} \frac{1}{K_j} \bbbclr{  \sum_{k=m+1}^{n} \mathbbm{1}_{ \left\{ X_k\st{1} \in A_{K_j}\st{1}(m) \right\} }  - \sum_{k=m+1}^{n} \mathbbm{1}_{ \left\{ X_k\st{2} \in A_{K_j}\st{2}(m) \right\} } }^2 > C\log(n),
       \]
       for some fixed  $C>0$.
    \end{enumerate}
    Then
    \ban{
        \frac{1}{(\ell + 1)\log(n)} \sum_{j=0}^{\ell} \frac{1}{K_j} &\bbbclr{ \sum_{x \in A_{K_j}\st{1}(n)} \deg_{\tree_n\st{1}}(x) -  \sum_{x \in A_{K_j}\st{2}(n)} \deg_{\tree_n\st{2}}(x)  }^2
           \label{eq:obsstat} \\
         &> C  - \frac{1}{(\ell + 1)\log(n)} \sum_{j=0}^{\ell} \frac{2(m+1) \log(n)^{2/3} + (m+1)^2}{K_j}. \notag
    }
\end{lemma}

\begin{lemma}\label{lem:sepstatdetconP}
Under the setup of Lemma~\ref{lem:sepstatdetcon}, assume that \ref{item:kcentralstat} and~\ref{item:kcentdegsmall} hold, and, in addition, that the following holds. 
    \begin{enumerate}[label=(\alph*'), start=3]
         \item\label{item:sepstatcritp} After time step $m$, the weighted sum of squares of the differences in the number of nodes connecting to $A_{K_j}\st{i}(m)$ in each of the two trees satisfies the bound
        \[
        \frac{1}{\ell + 1} \sum_{j=0}^{\ell} \frac{1}{K_j} \bbbclr{  \sum_{k=m+1}^{n} \mathbbm{1}_{ \left\{ X_k\st{1} \in A_{K_j}\st{1}(m) \right\} }  - \sum_{k=m+1}^{n} \mathbbm{1}_{ \left\{ X_k\st{2} \in A_{K_j}\st{2}(m) \right\} } }^2 < C\log(n),
       \]
       for some fixed  $C>0$.
    \end{enumerate}
    Then
    \ba{
        \frac{1}{(\ell + 1)\log(n)} \sum_{j=0}^{\ell} \frac{1}{K_j} &\bbbclr{ \sum_{x \in A_{K_j}\st{1}(n)} \deg_{\tree_n\st{1}}(x) -  \sum_{x \in A_{K_j}\st{2}(n)} \deg_{\tree_n\st{2}}(x)  }^2 \\
         &< C  + \frac{1}{(\ell + 1)\log(n)} \sum_{j=0}^{\ell} \frac{2(m+1) \log(n)^{2/3} + (m+1)^2}{K_j}. 
    }
\end{lemma}

A crucial aspect of the lemmas is that while the hypotheses involve the variables $X_{j}\st{i}$, which are latent in the hypothesis testing problem, the conclusion is about the observable statistic~\eqref{eq:obsstat}. 
Thus we can use the latent variables to estimate probabilities of the observable statistic~\eqref{eq:obsstat}.


\begin{proof}[Proof of Lemmas~\ref{lem:sepstatdetcon} and~\ref{lem:sepstatdetconP}]
The proofs of the two lemmas are almost identical, so we only prove Lemma~\ref{lem:sepstatdetcon}.
First note that for any $A\subseteq\{0,1,\ldots, n\}$, we have 
    \[
    \sum_{x \in A } \deg_{\tree_n\st{i}}(x)= |A\setminus \{0\}| +\sum_{k=1}^n \mathbbm{1}_{\{X_k\st{i} \in A\}}.
    \]
    Writing $r_j:=\babs{ A_{K_j}\st{1}(n) \setminus \{0\} } - \babs{ A_{K_j}\st{2}(n) \setminus \{0\} } \in \{-1,0,1\}$, 
    we then have that  
     \begin{equation}\label{eq:r}
     \sum_{x \in A_{K_j}\st{1}(n)} \deg_{\tree_n\st{1}}(x) -  \sum_{x \in A_{K_j}\st{2}(n)} \deg_{\tree_n\st{2}}(x) = r_j + \sum_{k=1}^{n} \mathbbm{1}_{ \left\{ X_k\st{1} \in A_{K_j}\st{1}(n) \right\} }  - \sum_{k=1}^{n} \mathbbm{1}_{ \left\{ X_k\st{2} \in A_{K_j}\st{2}(n) \right\} },
 \end{equation}
 and, because $r_j\in\{-1,0,1\}$, that also
    \begin{equation}\label{difference bound}
    \bbbabs{r_j+ \sum_{k=1}^{m} \mathbbm{1}_{ \left\{X_k\st{1} \in A_{K_j}\st{1}(m) \right\} }  - \sum_{k=1}^{m} \mathbbm{1}_{ \left\{X_k\st{2} \in A_{K_j}\st{2}(m)\right\} } } \leq m+1.
\end{equation}
Thus, we find that
    \ba{
     & \frac{1}{(\ell + 1)} \sum_{j=0}^{\ell} \frac{1}{K_j} \bbbclr{ \sum_{x \in A_{K_j}\st{1}(n)} \deg_{\tree_n\st{1}}(x) -  \sum_{x \in A_{K_j}\st{2}(n)} \deg_{\tree_n\st{2}}(x)  }^2
     \\
         & = \frac{1}{\ell + 1} \sum_{j=0}^{\ell} \frac{1}{K_j} \bbbclr{r_j+ \sum_{k=1}^{n} \mathbbm{1}_{ \left\{ X_k\st{1} \in A_{K_j}\st{1}(n) \right\} }  - \sum_{k=1}^{n} \mathbbm{1}_{ \left\{ X_k\st{2} \in A_{K_j}\st{2}(n) \right\} } }^2 
        \\
        & = \frac{1}{\ell + 1} \sum_{j=0}^{\ell} \frac{1}{K_j} \bbbclr{r_j+ \sum_{k=1}^{n} \mathbbm{1}_{ \left\{ X_k\st{1} \in A_{K_j}\st{1}(m) \right\} }  - \sum_{k=1}^{n} \mathbbm{1}_{ \left\{ X_k\st{2} \in A_{K_j}\st{2}(m) \right\} } }^2
        \\
        &
        \geq
        \frac{1}{\ell + 1} \sum_{j=0}^{\ell} \frac{1}{K_j} \bbbclr{ \sum_{k=m+1}^{n} \mathbbm{1}_{ \left\{X_k\st{1} \in A_{K_j}\st{1}(m) \right\} }  - \sum_{k=m+1}^{n} \mathbbm{1}_{ \left\{X_k\st{2} \in A_{K_j}\st{2}(m)\right\} } }^2
        \\
        &
        \qquad -
        \frac{1}{\ell + 1} \sum_{j=0}^{\ell} \frac{1}{K_j} \bbbclr{ 2(m+1) \bbbabs{ \sum_{k=m+1}^{n} \mathbbm{1}_{ \left\{X_k\st{1} \in A_{K_j}\st{1}(m) \right\} }  - \sum_{k=m+1}^{n} \mathbbm{1}_{ \left\{ X_k\st{2} \in A_{K_j}\st{2}(m)  \right\} } }+ (m+1)^2 }
        \\
        &
        > C \log(n)- \frac{1}{\ell + 1} \sum_{j=0}^{\ell} \frac{2 (m+1) \log(n)^{2/3} + (m+1)^2}{K_j},
    }
    where the first equality is because of \eqref{eq:r}, the second equality is because of \ref{item:kcentralstat}, the second to last inequality is because of \eqref{difference bound} and the elementary inequality $(x+y)^2 \geq x^2 - 2|xy| - y^2$, and the last inequality follows from~\ref{item:kcentdegsmall} and~\ref{item:sepstatcrit}. Dividing both sides by $\log(n)$ proves the result.
\end{proof}

To prove Lemma~\ref{lem:sep1}~(respectively Lemma \ref{lem:sep2}) we show that with good probability for $\UA$ trees, \ref{item:kcentralstat},  \ref{item:kcentdegsmall}, and  \ref{item:sepstatcrit} (respectively \ref{item:sepstatcritp})
hold for appropriate parameters. The last two items are handled by Lemma~\ref{lem:covariances}, while for~\ref{item:kcentralstat}, we use the following result, which is due to Jog and Loh \cite[Theorem 2]{Jog2018}.
\begin{theorem}[Theorem 2 in \cite{Jog2018}]\label{theo:jog loh}
    For all $\eps >0, K\in \N$ there exists $m \in \N$ such that for a uniform-attachment tree $(T_n)_{n\in \N}$ one has
    \begin{align*}
        \p\left(  A_K(T_m) = A_K(T_n) \text{ for all } n \geq m \right) > 1 - \eps .
    \end{align*}
\end{theorem}

We can now prove the first of the main results of this section.  
 \begin{proof}[Proof of Lemma \ref{lem:sep2}]
    Let $\ell$ be large enough so that
    \begin{equation}\label{ell condition}
        \frac{40}{ (\ell+1) (0.01 \alpha)^2 } \leq \frac{\eps}{5} .
    \end{equation}
For $K \in \N $, set $K_j:=K 2^j$, $j=0,\ldots, \ell$, and let  $ N\in\N$ and $K$ both be large enough so that for all $n \geq N$
    \begin{equation}\label{eq:kjcentshare}
        \sum_{j=0}^{\ell}\p \left( \left| A_{K_j}\st{1}(n) \cap A_{K_j}\st{2}(n) \right| < 0.2 \cdot K_j \right) \leq \frac{\eps}{5}.
    \end{equation}
    The existence of such $K,N \in \N$ follows from Corollary \ref{cor:intersection},
    since
    $
   \lim_{K\to\infty} \sum_{j=0}^\ell\frac{ \log(2^j K) }{2^j K} =0.$   Given those values of $\ell, K$, 
    let $ m\geq N$ be large enough so that
    \begin{equation}\label{eq:kjcentstat}
        \sum_{i=1}^{2} \sum_{j=0}^{\ell}\p \left( A_{K_j}\st{i}(M) \neq A_{K_j}\st{i}(m) \text{ for some } M\geq  m \right) < \frac{\eps}{5}.
    \end{equation}
    The existence of such an integer $ m\in \N$ follows from Theorem \ref{theo:jog loh}.
    Finally, take $n\geq m $ large enough so that
    \begin{equation} \label{m condition}
        \sum_{j=0}^{\ell} \frac{4 \cdot 2^j K}{\log(n)^{1/3}} < \frac{\eps}{5},
    \end{equation}
    and
    \begin{equation}\label{eq:nlargeno}
        \frac{2 (m+1) + (m+1)^2}{\log(n)^{1/3}} < 0.01 \alpha.
    \end{equation}
    In view of applying Lemma~\ref{lem:sepstatdetconP}, define the events 
    \begin{align*}
        \cA & = 
        \bigcap_{i=1}^{2} \bigcap_{j=0}^{\ell} \left\{ A_{K_j}\st{i}(M)= A_{K_j}\st{i}(m) \text{ for all } M\geq m \right\} ,
        \\
        \cB & = \bbbclc{\max_{j \in \{0,\ldots, \ell \}} \bbbabs{ \sum_{k=m+1}^{n} \mathbbm{1}_{ \left\{ X_k\st{1} \in A_{K_j}\st{1}(m) \right\} }  - \sum_{k=m+1}^{n} \mathbbm{1}_{ \left\{ X_k\st2 \in A_{K_j}\st{2}(m)\right\} } } \leq \log(n)^{2/3}},
        \\
        \cC & = \left\{ \frac{1}{\ell + 1} \sum_{j=0}^{\ell} \frac{1}{K_j \log(n)} \left(  \sum_{k=m+1}^{n} \mathbbm{1}_{ \left\{ X_k\st{1} \in A_{K_j}\st{1}(m) \right\} }  - \sum_{k=m+1}^{n} \mathbbm{1}_{ \left\{ X_k\st2 \in A_{K_j}\st{2}(m)\right\} } \right)^2 < 2-0.08 \alpha \right\},
    \end{align*}
    where we recall that $X_k\st{i}\sim \mathrm{Uniform}\{0,\ldots, k-1\}$ represents the vertex that vertex $k$ attaches to in $\tree_k\st{i}$.
    If the events $\cA, \cB, \cC$ hold, then
    Lemma~\ref{lem:sepstatdetconP} with $K_j=K 2^j$ and $C=2-0.08\alpha$
    implies 
    \ba{
 \frac{1}{(\ell + 1)\log(n)} \sum_{j=0}^{\ell} &\frac{1}{K_j} \left( \sum_{x \in A_{K_j}\st{1}(n)} \deg(x) -  \sum_{x \in A_{K_j}\st{2}(n)} \deg(x)  \right)^2\\
    & \leq 2-0.08 \alpha + \frac{1}{\ell + 1} \sum_{j=0}^{\ell} \frac{2(m+1) \log(n)^{2/3} + (m+1)^2}{2^j K \log(n)}\\
       &  \leq 2-0.08 \alpha +   \frac{2 (m+1) + (m+1)^2}{\log(n)^{1/3}} < 2-0.05 \alpha,
    }
    where the last inequality is because of~\eqref{eq:nlargeno}.
    Thus, in order to show Lemma~\ref{lem:sep2}, it suffices to show that $\p(\cA  \cap \cB \cap \cC) > 1 - \eps$, which is equivalent to saying that $\p(\cA^c   \cup \cB^c \cup \cC^c)<\eps$. 
    Introducing the event
    \ba{
    \cD & = \bigcap_{j=0}^{\ell} \left\{ \left| A_{K_j}\st{1}(m) \cap A_{K_j}\st{2}(m) \right| \geq 0.2 \cdot K_j  \right\},
    }
    a union bound implies 
    \begin{equation*}
        \p(\cA^c  \cup \cB^c \cup \cC^c ) \leq \p (\cA^c) + \p (\cB^c| \cD) + \p (\cC^c | \cD) + \p(\cD^c),
    \end{equation*}
    and we show that each of the four summands is bounded by $\frac{\eps}{5}$, which shows the result. The inequality $\p(\cA^c) \leq \frac{\eps}{5}$ directly follows from a union bound and \eqref{eq:kjcentstat}. The inequality $\p(\cD^c) \leq \frac{\eps}{5}$ directly follows from a union bound and~\eqref{eq:kjcentshare}, noting that $m\geq N$.

    For the third summand $\p (\cC^c | \cD)$, because $(X_k\st1,X_k\st2)_{k> m}$ are independent of $A_{K_j}\st{i}(m)$, $i=1,2$, (as functions only of the edge indicators $(X_k\st1,X_k\st2)_{k=1}^{m}$), we have that given $\cD$, the random variables 
    \ben{\label{eq:ustosxz}
    \bbclr{\sum_{k=m+1}^{n} \mathbbm{1}_{ \left\{ X_k
    \st1\in A_{K_j}\st{1}(m)\right\}}   - \sum_{k=m+1}^{n} \mathbbm{1}_{ \left\{ X_k\st2 \in A_{K_j}\st{2}(m)\right\}}}, \ \ j=0,\ldots,\ell
    } 
    have the same joint distributions as the random variables $\bclr{S_j^X(n)-S_j^Y(n)}$ defined at~\eqref{eq:sjxyz} with $M_{K_j} = A_{K_j}\st{1}(m)$, $\wt M_{K_j}= A_{K_j}\st{2}(m)$, and $N=m$, and satisfy~\eqref{eq:gdcondvars}. 
    Thus,~\eqref{eq:sepstatprbind} of Lemma~\ref{lem:covariances} implies 
    \be{
    \P(\cC^c | \cD) \leq \frac{20 \sum_{0\leq j \leq i \leq \ell}2^{j-i}}{(\ell+1)^2 (0.01\alpha)^2 }\leq \frac{40}{(\ell+1) (0.01\alpha)^2 }\leq \frac{\eps}{5},
    }
    where we have used $K_j = K 2^j$ in the first inequality, an easy direct calculation in the second, and~\eqref{ell condition} in the third.
    
Using again that conditional on $\cD$, we may apply Lemma~\ref{lem:covariances} to~\eqref{eq:ustosxz},  a union bound and~\eqref{eq:absdiffind} with $t=\log(n)^{2/3}$  imply 
\be{
\P(\cB^c| \cD)\leq \frac{4}{\log(n)^{1/3}}\sum_{j=0}^\ell K 2^j \leq \frac{\eps}{5},
}
where we have used~\eqref{m condition} and that $n\geq m$ in the second inequality.
\end{proof}

\begin{proof}[Proof of Lemma~\ref{lem:sep1}]
The proof follows nearly the same way as the previous one, but we replace the event $\cC$ with 
\be{
\cC'  = \left\{ \frac{1}{\ell + 1} \sum_{j=0}^{\ell} \frac{1}{K_j \log(n)} \left(  \sum_{k=m+1}^{n} \mathbbm{1}_{ \left\{ X_k\st1 \in A_{K_j}\st{1}(m) \right\} }  - \sum_{k=m+1}^{n} \mathbbm{1}_{ \left\{ X_k^2 \in A_{K_j}\st{2}(m)\right\} } \right)^2 > 2 - 0.02 \alpha \right\},
}
and we use~\eqref{eq:variance lem 1},~\eqref{eq:variance lem 3},~\eqref{eq:sepstatprbind}, and~\eqref{eq:absdiffind} where we previously used~\eqref{eq:variance lem 2},~\eqref{eq:variance lem 4},~\eqref{eq:sepstatprbcor}, and~\eqref{eq:absdiffcor} (respectively). 
\end{proof}

\section{Estimating the correlation parameter $\alpha$}\label{sec:genres}

In this section, we prove Corollary~\ref{cor:main}.
Recall from~\eqref{eq:uhatdef}, that for a vertex $v\in T_n^i$ and $j \leq \ell \in \N$, we define
\begin{align*}
    &\widehat{U}_{\ell,j}^{T_n^i}(u) = \sum_{v\in \tree_n^i} \mathbbm{1} \left\{ v\sim_{i} u , \left| K_v \left( T_n^i \setminus \{u\} \right) \right| \in \left( \frac{n}{e^j} , \frac{n}{e^\ell} \right] \right\},
\end{align*}
where for a tree $T_n^i$ and two distinct vertices $u,v$ in the tree, we write $\left| K_v \left( T_n^i \setminus \{u\} \right) \right|$ for the size of the tree that contains $v$ when removing $u$. 

The theorem directly follows from the following lemma.
\begin{lemma}\label{lem:sep3}
    For all $\eps>0$ there exist $m,k,K \in \N$ such that for all $n \geq m$, whenever $\left( T_n^1, T_n^2 \right)$ is distributed according to $\alphaCUA$, then
    \ba{
        \p \left( \left| \min_{u_1 \in A_K^1(n), u_2 \in A_K^2(n)} \frac{1}{k^3} \sum_{j=1}^{k} \left( \widehat{U}_{jk^2,(j+1)k^2}^{T_n^1}(u_1) - \widehat{U}_{jk^2,(j+1)k^2}^{T_n^2}(u_2) \right)^2 - 2(1-\alpha) \right| \leq \eps  \right) > 1-\eps.
    }
\end{lemma}
As mentioned previously, we study the $\wh U^{T_n^i}$ variables through the easier to understand quantities
\besn{\label{eq:Ukey}
         U_{\ell,j}^{T_n^i}(u)& := \sum_{v=M+1}^{n} \mathbbm{1} \left\{ v\sim_{i} u , F_{T_n^i} (v) \in \left( \frac{n}{e^j} , \frac{n}{e^\ell} \right] \right\}.
    }
We will eventually show that the first order behavior of our estimator is the same when  
 replacing the $\wh U^{T_n^i}$ variables in  it with the $U^{T_n^i}$. 
The next lemma is the key to our proof, and  shows that to first order, $U_{\ell,j}^{T_n^i}(u)$ behaves as the simpler
\besn{\label{eq:Dkey}
D_{\ell,j}^{T_n^i}(u) := \sum_{v=e^\ell}^{e^j - 1} \mathbbm{1} \left\{ v\sim_{i} u \right\},
}
which 
is a sum of independent Bernoulli-distributed random variables. 
\begin{lemma}\label{lem:zmombd}
Let $\ell, j, M, n, u$ be such that $u< M+1 \leq e^\ell<e^j \leq n$, and recall the definitions at~\eqref{eq:Ukey} and~\eqref{eq:Dkey}.
There is a universal constant $C_\star$   such that the random variable
\be{
Z_{\ell,j}^{T_n^i}(u):= U_{\ell,j}^{T_n^i}(u) - D_{\ell,j}^{T_n^i}(u)  
}
satisfies 
    \be{
    \E\bcls{ Z_{\ell,j}^{T_n^i}(u)^4}\leq C_\star.
    }
\end{lemma}
The proof of the lemma roughly follows by decomposing $Z_{\ell,j}^{T_n^i}(u)$ into sums and differences of six sums of Bernoulli variables, each which has a uniformly bounded mean, and each of which can be usefully compared to a sum of independent Bernoulli variables with the same mean, for which the fourth moment can be compared to the first (see Proposition~\ref{claim:moment expansion}). For the sake of exposition, we postpone the details to 
Section~\ref{sec:zmombdpf}.

Armed with this lemma, we have the following moment estimates for a statistic built from the $U^{T_n^i}$'s, which implies adequate concentration. 

\begin{proposition}\label{prop:moment variance X}
   For $u_1 , u_2 \in \{0,\ldots,M\}$ and $k\in \N$ with $M+1 \leq e^{k^2}$ and $n \geq e^{(k+1)k^2}$,  define the random variable 
    \begin{align*}
        X_k(u_1, u_2) = \frac{1}{k} \sum_{j=1}^{k} \frac{1}{k^2} \left( U_{jk^2,(j+1)k^2}^{T_n^1}(u_1) - U_{jk^2,(j+1)k^2}^{T_n^2}(u_2) \right)^2.
    \end{align*}
    Then, for each fixed $\eps > 0$ and for $k$ large enough,
    \begin{equation}\label{eq:xkmean}
        \E \left[ X_k(u_1,u_2)\right] \in \left( 2\left(1 - \alpha \mathbbm{1}_{u_1 = u_2}\right) - \frac{\eps}{10}, 2\left(1 - \alpha \mathbbm{1}_{u_1 = u_2}\right) + \frac{\eps}{10} \right)
    \end{equation}
    and
    \begin{equation}\label{eq:xkvar}
        \Var \left( X_k(u_1,u_2) \right) \leq \frac{\Bar{C}}{k} 
    \end{equation}
    for some $\Bar{C} < \infty$.
\end{proposition}

\begin{proof}
    Recalling that $U_{jk^2,(j+1)k^2}^{T_n^1}(u) = D_{jk^2,(j+1)k^2}^{T_n^1}(u) + Z_{jk^2,(j+1)k^2}^{T_n^1}(u)$ for any $u$, and, for the sake of brevity, writing
    \begin{align*}
        \Delta U_{j k^2,(j+1)k^2}^{(u_1,u_2)} & \coloneqq U_{j k^2,(j+1)k^2}^{T_n^1}(u_1) - U_{j k^2,(j+1)k^2}^{T_n^2}(u_2),
        \\
        \Delta D_{j k^2,(j+1)k^2}^{(u_1,u_2)} & \coloneqq D_{j k^2,(j+1)k^2}^{T_n^1}(u_1) - D_{j k^2,(j+1)k^2}^{T_n^2}(u_2),
        \\
        \Delta Z_{j k^2,(j+1)k^2}^{(u_1,u_2)} & \coloneqq Z_{j k^2,(j+1)k^2}^{T_n^1}(u_1) - Z_{j k^2,(j+1)k^2}^{T_n^2}(u_2),
    \end{align*}
    we have
    that
    \ban{
          \E \left[ \frac{1}{k^2} \left( \Delta U_{j k^2,(j+1)k^2}^{(u_1,u_2)} \right)^2 \right]
        & 
        =
        \E \left[ \frac{1}{k^2} \left(  \Delta D_{j k^2,(j+1)k^2}^{(u_1,u_2)}  + \Delta Z_{j k^2,(j+1)k^2}^{(u_1,u_2)} \right)^2 \right]
        \notag \\
        \begin{split}
        &=
        \E \left[ \frac{1}{k^2} \left( \Delta D_{j k^2,(j+1)k^2}^{(u_1,u_2)}\right)^2 \right] + \E \left[ \frac{1}{k^2} \left( \Delta Z_{j k^2,(j+1)k^2}^{(u_1,u_2)} \right)^2 \right]
        \\
        & 
        \qquad\qquad +
        \frac{2}{k^2} \E \left[ \Delta D_{j k^2,(j+1)k^2}^{(u_1,u_2)}  \,  \Delta Z_{j k^2,(j+1)k^2}^{(u_1,u_2)}  \right].
        \end{split} \label{eq:sum 3 terms}
    }
    For the first term in this sum, for all $k$ large enough and all $j \in [k]$, we have that 
    \besn{\label{eq:xkmt1}
         \E \left[ \frac{1}{k^2} \left( \Delta D_{j k^2,(j+1)k^2}^{(u_1,u_2)} \right)^2 \right] 
         &= \frac{1}{k^2} \E \left[ \left( \sum_{v= e^{jk^2}}^{e^{(j+1)k^2}-1} \left( \mathbbm{1}_{\left\{ v \sim_{1} u_1 \right\}} - \mathbbm{1}_{\left\{ v \sim_{2} u_2 \right\}} \right) \right)^2 \right]
         \\
        &
        =
        \frac{1}{k^2} \sum_{v= e^{jk^2}}^{e^{(j+1)k^2}-1} \left( \frac{2\left(1 - \alpha \mathbbm{1}_{u_1 = u_2}\right)}{v} - \frac{2(1-\alpha)}{v^2} \right)\\
        &\in \left( 2\left(1 - \alpha \mathbbm{1}_{u_1 = u_2}\right) \pm \frac{\eps}{20}\right),
    }
    where the second equality follows by  a direct calculation, using that the mean zero random variables $\left( \mathbbm{1}_{\left\{ v \sim_{1} u_1 \right\}} - \mathbbm{1}_{\left\{ v \sim_{2} u_2 \right\}} \right)$ are independent across $v$.
     Also note that the above calculation directly implies that for $k$ large enough, we have
    \begin{align}\label{eq:M bound}
        \E \left[ \left(\Delta D_{j k^2,(j+1)k^2}^{(u_1,u_2)}\right)^2 \right] \leq 3k^2
    \end{align}
    For the second summand in \eqref{eq:sum 3 terms}, Lemma~\ref{lem:zmombd} combined with the Cauchy-Schwarz inequality implies that for $k$ large enough, we have
    \besn{\label{eq:xkmt2}
        \E \left[ \frac{1}{k^2} \left(\Delta Z_{j k^2,(j+1)k^2}^{(u_1,u_2)} \right)^2 \right]  \leq \frac{2}{k^2} \E \left[  Z_{jk^2,(j+1)k^2}^{T_n^1}(u_1)^2 + Z_{jk^2,(j+1)k^2}^{T_n^2}(u_2)^2 \right] \leq \frac{4 \sqrt{C_\star}}{k^2} \leq \frac{\eps}{100}.
    }
    For the third term in \eqref{eq:sum 3 terms}, we use the Cauchy-Schwarz inequality, Lemma~\ref{lem:zmombd} combined with the Cauchy-Schwarz inequality, and~\eqref{eq:M bound}, to find that for $k$ large enough 
    \besn{\label{eq:xkmt3}
    \left| \frac{2}{k^2} \E \left[ \Delta D_{j k^2,(j+1)k^2}^{(u_1,u_2)} \, \Delta Z_{j k^2,(j+1)k^2}^{(u_1,u_2)} \right] \right|     
        &
        \leq
        \frac{2}{k^2} \E \left[ \left( \Delta D_{j k^2,(j+1)k^2}^{(u_1,u_2)} \right)^2 \right]^{1/2} \E \left[ \left(\Delta Z_{j k^2,(j+1)k^2}^{(u_1,u_2)} \right)^2 \right]^{1/2}
        \\
        &
        \leq 
        \frac{2}{k^2} \left(3k^2 \right)^{1/2} \left( 4 \sqrt{C^\star} \right)^{1/2} \leq \frac{\eps}{100}.
    }
 Inserting~\eqref{eq:xkmt1},~\eqref{eq:xkmt2}, and~\eqref{eq:xkmt3} into \eqref{eq:sum 3 terms}, we find
    \begin{align*}
        \E \left[ \frac{1}{k^2} \left( \Delta U_{j k^2,(j+1)k^2}^{(u_1,u_2)} \right)^2 \right] \in \left(2\left(1 - \alpha \mathbbm{1}_{u_1 = u_2}\right) - \frac{\eps}{10}, 2\left(1 - \alpha \mathbbm{1}_{u_1 = u_2}\right)+ \frac{\eps}{10} \right),
    \end{align*}
which gives the desired result~\eqref{eq:xkmean} for the mean of $X_k(u_1,u_2)$.
    
    To bound the variance  of $X_k(u_1,u_2)$, we need to bound covariance terms. We have
    \begin{align*}
        \Cov \bbbclr{\bbclr{ &\Delta U_{j k^2,(j+1)k^2}^{(u_1,u_2)} }^2,  \bbclr{ \Delta U_{j^\prime k^2,(j^\prime+1)k^2}^{(u_1,u_2)}}^2}
        \\
        &
        =
        \Cov \left( \left( \Delta D_{j k^2,(j+1)k^2}^{(u_1,u_2)} \right)^2,  \left( \Delta D_{j^\prime k^2,(j^\prime+1)k^2}^{(u_1,u_2)} \right)^2 \right) 
        \\
        &
        \hspace{5mm}+ 
        \Cov \left( \left( \Delta D_{j k^2,(j+1)k^2}^{(u_1,u_2)} \right)^2,  \left( \Delta Z_{j^\prime k^2,(j^\prime+1)k^2}^{(u_1,u_2)} \right)^2 \right)
        \\
        & \hspace{5mm}+ 
        2
        \Cov \left( \left( \Delta D_{j k^2,(j+1)k^2}^{(u_1,u_2)} \right)^2,  \Delta D_{j^\prime k^2,(j^\prime+1)k^2}^{(u_1,u_2)} \Delta Z_{j^\prime k^2,(j^\prime+1)k^2}^{(u_1,u_2)} \right)
        \\
        &
        \hspace{5mm}+  2
        \Cov \left( \Delta D_{j k^2,(j+1)k^2}^{(u_1,u_2)} \Delta Z_{j k^2,(j+1)k^2}^{(u_1,u_2)} ,  \left( \Delta D_{j^\prime k^2,(j^\prime+1)k^2}^{(u_1,u_2)} \right)^2 \right) 
        \\
        &\hspace{5mm}
        + 
        2 \Cov \left( \Delta D_{j k^2,(j+1)k^2}^{(u_1,u_2)} \Delta Z_{j k^2,(j+1)k^2}^{(u_1,u_2)} ,  \left( \Delta Z_{j^\prime k^2,(j^\prime+1)k^2}^{(u_1,u_2)} \right)^2 \right)
        \\
        & \hspace{5mm}+ 
        4
        \Cov \left( \Delta D_{j k^2,(j+1)k^2}^{(u_1,u_2)} \Delta Z_{j k^2,(j+1)k^2}^{(u_1,u_2)} ,  \Delta D_{j^\prime k^2,(j^\prime+1)k^2}^{(u_1,u_2)} \Delta Z_{j^\prime k^2,(j^\prime+1)k^2}^{(u_1,u_2)} \right)
        \\
        &\hspace{5mm}
        +
        \Cov \left( \left( \Delta Z_{j k^2,(j+1)k^2}^{(u_1,u_2)} \right)^2,  \left( \Delta D_{j^\prime k^2,(j^\prime+1)k^2}^{(u_1,u_2)} \right)^2 \right) 
        \\
        &\hspace{5mm}
        + 
        \Cov \left( \left( \Delta Z_{j k^2,(j+1)k^2}^{(u_1,u_2)} \right)^2,  \left( \Delta Z_{j^\prime k^2,(j^\prime+1)k^2}^{(u_1,u_2)} \right)^2 \right)
        \\
        & \hspace{5mm}+
        2
        \Cov \left( \left( \Delta Z_{j k^2,(j+1)k^2}^{(u_1,u_2)} \right)^2,  \Delta D_{j^\prime k^2,(j^\prime+1)k^2}^{(u_1,u_2)} \Delta Z_{j^\prime k^2,(j^\prime+1)k^2}^{(u_1,u_2)} \right).
    \end{align*}
    We now bound each of the $9$ terms in the above sum separately. The first term is $0$ for $j\neq j^\prime$, as the relevant random variables are independent. 
    For $j=j^\prime$,
    \besn{\label{eq:deldpbd}
    \Var\left(\left( \Delta D_{j k^2,(j+1)k^2}^{(u_1,u_2)} \right)^4\right)
    &\leq \E\left[ \left( \Delta D_{j k^2,(j+1)k^2}^{(u_1,u_2)} \right)^4 \right] \\
    &= \E \left[ \left( \sum_{v=e^{jk^2}}^{e^{(j+1)k^2}-1} \mathbbm{1}\left\{ v \sim_1 u_1 \right\} - \mathbbm{1}\left\{ v \sim_2 u_2 \right\} \right)^4 \right]
        \\
        &
        \leq 
        3 \left( \sum_{v=e^{jk^2}}^{e^{(j+1)k^2}-1} \frac{2}{v} \right)^2 +  \left( \sum_{v=e^{jk^2}}^{e^{(j+1)k^2}-1} \frac{2}{v} \right)\\
        &\leq 36 k^4,
    }
    where the second inequality follows by direct calculation, using that for centered independent random variables $(Y_\gamma)_\gamma$ with $Y_\gamma\in\{-1,0,1\}$, we have 
    \ba{
      &\E\bbbcls{ \bbclr{\sum_{\gamma}Y_\gamma}^4 } = 3 \sum_{\substack{\gamma,\gamma' : \\ \gamma \neq \gamma'}} \E\left[ Y_{\gamma}^2 Y_{\gamma'}^2 \right] + \sum_{\gamma} \E[Y_\gamma^4] \leq 3 \left(\sum_{\gamma} \E\left[ Y_{\gamma}^2\right]\right)^2  + \sum_{\gamma} \E[Y_\gamma^4],
      }
      and that $\E\bcls{\clr{\mathbbm{1}\left\{ v \sim_1 u_1 \right\} - \mathbbm{1}\left\{ v \sim_2 u_2 \right\}}^2}\leq 2/v$, read from~\eqref{eq:xkmt1}.

    The remaining $8$ terms are bounded using that 
    \be{
    \Cov(X,Y) \leq\sqrt{\Var(X) \Var(Y)}\leq \sqrt{\E\left[X^2 \right] \E\left[Y^2 \right]},
    }
and, 
    from~\eqref{eq:deldpbd}, that $\E\bbcls{\bclr{ \Delta D_{j k^2,(j+1)k^2}^{(u_1,u_2)} }^4 }\leq 36k^4$, and finally, 
     from Lemma~\ref{lem:zmombd}, that for some universal constant $C_\star>0$, we have
    \begin{align*}
        \E \left[ \left(\Delta Z_{i k^2,(i+1)k^2}^{(u_1,u_2)} \right)^4 \right] \leq C_\star.
    \end{align*}
  This implies that the remaining eight terms are upper bounded of order $k^3$, and so we have that
    \begin{align*}
        \Cov \left( \left( \Delta U_{j k^2,(j+1)k^2}^{(u_1,u_2)} \right)^2,  \left( \Delta U_{j^\prime k^2,(j^\prime+1)k^2}^{(u_1,u_2)} \right)^2 \right)
        \leq 36 k^4 \mathbbm{1}_{\{j=j^\prime\}} + \widetilde{C} k^3
    \end{align*}
    for some constant $\widetilde{C} < \infty$. 
    
    Thus we find that
    \begin{align*}
         \Var\left( X_k(u_1,u_2) \right) 
        &=
        \frac{1}{k^6} \sum_{j=1}^{k} \sum_{j^\prime=1}^{k}
        \Cov \left( \left( \Delta U_{j k^2,(j+1)k^2}^{(u_1,u_2)} \right)^2 , \left( \Delta U_{j^\prime k^2,(j^\prime +1) k^2}^{(u_1,u_2)} \right)^2 \right)
        \\
        &
        \leq
        \frac{1}{k^6} \sum_{j=1}^{k} \sum_{j^\prime=1}^{k} \left( 36 k^4 \mathbbm{1}_{\{j=j^\prime\}} + \widetilde{C} k^3 \right) \leq \frac{36k^5 + \widetilde{C} k^5}{k^6} \leq \frac{\Bar{C}}{k},
    \end{align*}
    as desired.
\end{proof}

We can now prove our main result by linking the behavior of $X_k(u_1,u_2)$  to that of our observable statistic, and then applying Proposition~\ref{prop:moment variance X}.

\begin{proof}[Proof of Lemma~\ref{lem:sep3}]
Fix $\eps>0$.
We first define ``good'' events, under which the event considered in the lemma holds deterministically, and then lower bound the probability of these good events. 
For the latter, 
let $K \in \N$  be such that 
\begin{align}
    \p\left( A_K^1(n)\cap A_K^2(n)\not=\emptyset \right) > 1 - \frac{\eps}{10} \label{eq:evah}
\end{align}
for all $n\geq N(K)$ large enough.
The existence of such $K$ is guaranteed by Corollary~\ref{cor:intersection}.
Given this $K$, let  $M\in \N$ be such that $M \geq N(K)$ and 
 for all $n \geq M$, we have
    \begin{equation}\label{eq:evbh}
        \p\left( A_K^1(M) = A_K^1(\ell) \text{ for all } \ell \geq M \right) > 1-\frac{\eps}{10}.
    \end{equation}
 The existence of such an $M$ follows from \cite[Theorem~2]{Jog2018}, stated as Theorem~\ref{theo:jog loh} above.
Given $K$ and $M$, 
take $k$ large enough, and $n\geq e^{k^2(k+1)}$ also large enough so that
\begin{align}\label{eq:evch}
    (M+1)^2 \frac{100 \Var\left( X(u_1,u_2 ) \right)}{\eps^2} \leq \frac{\eps}{5},
\end{align} 
which can be done by Proposition~\ref{prop:moment variance X}. Finally, since 
$M$ is fixed and $F_{T_n^1}(M)/n \to F\sim\mathrm{Beta}(1,M)$ almost surely,  we can enlarge $k$ and $n\geq e^{k^2(k+1)}$ if necessary to make
\be{
2(M+1)^2\p (F\leq e^{-k}) \leq \frac{\eps}{10},
}
and then choose $n$ large enough so that
\be{
2(M+1)^2\bbbabs{\p (F\leq e^{-k})-\p\bbclr{F_{T_n^1}(M)\leq \frac{n}{e^k}}}\leq \frac{\eps}{10}
}
which implies 
\begin{align*}
    2(M+1)^2\p \left( F_{T_n^1}(M) \leq \frac{n}{e^k} \right) \leq \frac{\eps}{5}.
\end{align*}

Given $K,M,k$, and $n$, define the good events
\begin{align*}
    & \cA = \left\{  A_K^1\left(n\right) \cap A_K^2\left(n\right) \not= \emptyset \right\}, \\
    & \cB = \left\{ A_K^i(n) = A_K^i(M) , i=1,2 \right\}, \\
    & \cC = \bigcap_{u_1 \in A_K^1(M)} \bigcap_{u_2 \in A_K^2(M)} \left\{ \left| X(u_1,u_2) - \E \left[X(u_1,u_2) \right] \right| \leq \frac{\eps}{10} \right\}, \\
    & \cD = \bigcap_{i=1}^{2} \bigcap_{v=0}^{M} \bigcap_{\substack{ u = 0 \\ u \neq v}}^{M} \left\{ \left| K_v \left( T_n^i \setminus \{u\} \right)  \right| > \frac{n}{e^k}\right\}.
\end{align*}
With this setup, we  make the following \textbf{Claim:}  If the four events $\cA, \cB, \cC$, and $\cD$ hold, then
    \begin{align*}
        \min_{u_1 \in A_K^1(n), u_2 A_K^2(n)} \frac{1}{k^3} \sum_{j=1}^{k} \left( \widehat{U}_{jk^2,(j+1)k^2}^{T_n^1}(u_1) - \widehat{U}_{jk^2,(j+1)k^2}^{T_n^2}(u_2) \right)^2 \in (2(1-\alpha) - \eps, 2(1-\alpha) + \eps).
    \end{align*}
To see why the claim is true, if the event $\cB$ holds, then the minimum over $u_1 \in A_K^1(n), u_2 \in A_K^2(n)$ goes only over a subset of $\{0,\ldots,M\}$. 
Further, if the event $\cD$ holds, then \begin{align}\label{eq:uhatequ}
        \widehat{U}_{jk^2,(j+1)k^2}^{T_n^i}(u_i) = U_{jk^2,(j+1)k^2}^{T_n^i}(u_i).
    \end{align}
    This is because all vertices $v \in \{0,\ldots,M\}$ with $v\neq u_i$ satisfy $ \left| K_v \left( T_n^i \setminus \{u_i\} \right)  \right| > \frac{n}{e^{jk^2}}$, and thus do not contribute to the sum in $ \widehat{U}_{jk^2,(j+1)k^2}^{T_n^i}(u_i)$. For other vertices $u_i \leq M < v$, $|K_v \left( T_n^i \setminus \{u_i\} \right) | =F_{T_n^i}(v)$, which establishes~\eqref{eq:uhatequ}.
    Thus we see that if $\cB$ and $\cD$ hold, then 
    \begin{align*}
        \min_{u_1 \in A_K^1(n), u_2 \in A_K^2(n)} \frac{1}{k^3} &\sum_{j=1}^{k} \left( \widehat{U}_{jk^2,(j+1)k^2}^{T_n^1}(u_1) - \widehat{U}_{jk^2,(j+1)k^2}^{T_n^2}(u_2) \right)^2 \\
        &
        =
        \min_{u_1 \in A_K^1(M), u_2 \in A_K^2(M)} \frac{1}{k^3} \sum_{j=1}^{k} \left( \widehat{U}_{jk^2,(j+1)k^2}^{T_n^1}(u_1) - \widehat{U}_{jk^2,(j+1)k^2}^{T_n^2}(u_2) \right)^2
        \\
        &
        =
        \min_{u_1 \in A_K^1(M), u_2 \in A_K^2(M)} \frac{1}{k^3} \sum_{j=1}^{k} \left( U_{jk^2,(j+1)k^2}^{T_n^1}(u_1) - U_{jk^2,(j+1)k^2}^{T_n^2}(u_2) \right)^2
        \\
        &
        =
        \min_{u_1 \in A_K^1(M), u_2 \in A_K^2(M)} X(u_1, u_2). 
    \end{align*}
    If the events $\cA$ and $\cB$ hold, then $A_K^1(M) \cap A_K^2(M)  \neq \emptyset$ and  the minimum of $\E\left[ X(u_1, u_2 ) \right]$ over $u_1 \in A_K^1(M) , u_2 \in A_K^2(M)$ lies in $\left( 2(1-\alpha) - \frac{\eps}{10} , 2(1-\alpha) + \frac{\eps}{10} \right)$, by Proposition \ref{prop:moment variance X}. 
    If in addition the event $\cC$ holds, then also the minimum of $ X(u_1, u_2 )$ over $u_1 \in A_K^1(M)  , u_2 \in A_K^2(M) $ lies in $\left( 2(1-\alpha) - \frac{2\eps}{10} , 2(1-\alpha) + \frac{2\eps}{10} \right)$, which finishes the proof of the claim.

Since we have shown the claim is true, the result follows if we can show that for our choice of $M,n,k,K$, we have
    \begin{align*}
        \p \left(\cA^c \right) + \p \left(\cB^c \right) + \p \left(\cC^c \right) + \p \left(\cD^c \right) < \eps .
    \end{align*}
    By our choice of $K,M,n$,~\eqref{eq:evah} and~\eqref{eq:evbh} with a union bound imply 
    \begin{align*}
        \p \left(\cA^c \right) & \leq \frac{\eps}{10} \, \mbox{ and }  \,\, \p(\cB^c) \leq \frac{\eps}{5}.
        \end{align*}
   A union bound, Chebyshev's inequality, and~\eqref{eq:evch} imply 
    \begin{align*}
        \p \left( \cC^c \right) & \leq \sum_{u_1 \in \{0,\ldots,M\}} \sum_{u_2 \in \{0,\ldots,M\}} \p \left( \left| X(u_1,u_2) - \E \left[X(u_1,u_2) \right] \right| > \frac{\eps}{10} \right) 
        \\
        &
        \leq
        \sum_{u_1 \in \{0,\ldots,M\}} \sum_{u_2 \in \{0,\ldots,M\}} \frac{100 \Var(X(u_1,u_2))}{\eps^2}
        < \frac{\eps}{5} .
    \end{align*}
    For each pair $u,v \in \{0,\ldots,M\}$ with $u \neq v$, the size of $K_v \left(T_n^i \setminus \{u\}\right)$ stochastically dominates $F_{T_n^i}(M)$, so that a union bound implies 
    \begin{align*}
        \p \left( \cD^c \right) \leq  \sum_{i=1}^{2} \sum_{v=0}^{M} \sum_{\substack{ u = 0 \\ u \neq v}}^{M} \p \left( \left| K_v \left( T_n^i \setminus \{u\} \right)  \right| \leq \frac{n}{e^k}\right) 
        \leq
        2 (M+1)^2 \p \left( F_{T_n^i}(M) \leq \frac{n}{e^k} \right) < \frac{\eps}{5},
    \end{align*}
    which finishes the proof.
\end{proof}

\subsection{Proof of Lemma~\ref{lem:zmombd}}\label{sec:zmombdpf}

We first define some intermediate random variables, which can be used to decompose~$Z_{\ell,j}^{T_n^i}(u)$. \\
\vspace{-4mm}
\be{
\begin{minipage}{0.45\linewidth}
\begin{align*}
 O_{\ell,j}^{T_n^i}(u) & := \sum_{v=e^\ell}^{e^j-1} \mathbbm{1} \left\{ v\sim_{i} u , F_{T_n^i} (v) \leq \frac{n}{e^j}  \right\}, \\ 
        I_{\ell,j}^{1, T_n^i}(u) &:= \sum_{v=M+1}^{e^\ell-1} \mathbbm{1} \left\{ v\sim_{i} u , F_{T_n^i} (v) \leq \frac{n}{e^\ell}  \right\}, \\
        I_{\ell,j}^{2, T_n^i}(u) & := \sum_{v=e^j}^{n} \mathbbm{1} \left\{ v\sim_{i} u , F_{T_n^i} (v) > \frac{n}{e^j} \right\}, 
\end{align*}
\end{minipage}
\begin{minipage}{0.45\linewidth}
\begin{align*}
        \widetilde{O}_{\ell,j}^{T_n^i}(u)& := \sum_{v=e^\ell}^{e^j-1} \mathbbm{1} \left\{ v\sim_{i} u , F_{T_n^i} (v) >  \frac{n}{e^\ell} \right\}\\ 
        \widetilde{I}_{\ell,j}^{1, T_n^i}(u) & := \sum_{v=M+1}^{e^\ell-1} \mathbbm{1} \left\{ v\sim_{i} u , F_{T_n^i} (v) \leq \frac{n}{e^j}  \right\}, \\
        \widetilde{I}_{\ell,j}^{2 , T_n^i}(u) &:= \sum_{v=e^j}^{n} \mathbbm{1} \left\{ v\sim_{i} u , F_{T_n^i} (v) >  \frac{n}{e^\ell} \right\}.
\end{align*}
\end{minipage}
}
By 
 basic set operations, and separating summands according to the label of vertices $v\in [M+1,n]$, (i.e.,  $O, \wt O$ are terms  with $v\in [e^\ell, e^j)$, $I^1, \wt I^1$ have $v\in [M+1, e^\ell)$ and $I^2, \wt I^2$ have $v\in [e^j, n]$), we easily find that
\besn{\label{eq:zdecomp}
    Z_{\ell,j}^{T_n^i}(u)  & = - \left( O_{\ell,j}^{T_n^i}(u) + \widetilde{O}_{\ell,j}^{T_n^i}(u) \right) + \left( I_{\ell,j}^{1,T_n^i}(u) - \widetilde{I}_{\ell,j}^{1,T_n^i}(u) \right) + \left( I_{\ell,j}^{2,T_n^i}(u) - \widetilde{I}_{\ell,j}^{2,T_n^i}(u) \right).
}
Since for general $x_1,\ldots,x_6\in\R$, we have
\be{
(x_1+\cdots + x_6)^4\leq 6^4 \max_{i=1,\ldots, 6} x_i^4\leq 6^4 \sum_{i=1}^6 x_i^4, 
}
it is sufficient to show that the fourth moment of each of the six random variables in the decomposition~\eqref{eq:zdecomp} of $Z_{\ell,j}^{T_n^i}(u)$ is uniformly bounded. 
To accomplish this, we need the following result for sums of Bernoulli random variables.

\begin{proposition}\label{claim:moment expansion}
    Let $X_1,\ldots,X_n$ be Bernoulli random variables and let $X=\sum_{i=1}^{n} X_i$. Suppose that for any subset of indices $S\subseteq\{1,\ldots,n\}$ with $\abs{S}\leq 4$, there is a constant $\wt C$ such that
    \ben{\label{eq:negcor}
        \p \left( X_i = 1 \text{ for all } i \in S \right) \leq \wt C \prod_{i \in S} \p (X_i = 1).
    }
 Then there is a constant $\wh C$ depending only on $\wt C$ such that
    \begin{align*}
        \E \left[ X^4 \right] \leq \wh C \bclr{\E[X]+\E[X]^4}.
    \end{align*}
\end{proposition}
\begin{proof}
    Expanding the fourth power of $X$ and using that $X_i^k=X_i\in\{0,1\}$ for any $k>0$, there is a universal (combinatorial) constant $c_4$ such that
    \begin{align*}
        \E \left[ X^4 \right] & = \sum_{i_1,i_2,i_3,i_4=1}^{n} \E\left[ X_{i_1} X_{i_2} X_{i_3} X_{i_4} \right] \\
        &  
        \leq 
        c_4\bbbclr{ \sum_{i_1=1}^n \E\left[ X_{i_1} \right]
        +
        \sum_{\substack{i_1, i_2 \in \{1,\ldots,n\} :\\ i_1 \neq i_2}} \E\left[ X_{i_1}  X_{i_2} \right]
        \\
        &
        \qquad \qquad+
\sum_{\substack{i_1, i_2 , i_3 \in \{1,\ldots,n\} :\\ |\{i_1,i_2,i_3\}|=3}} \E\left[ X_{i_1}  X_{i_2}  X_{i_3} \right]
+
\sum_{\substack{i_1, i_2 , i_3, i_4 \in \{1,\ldots,n\} :\\ |\{i_1,i_2,i_3,i_4\}|=4}}
\E\left[ X_{i_1} X_{i_2} X_{i_3} X_{i_4} \right]}
        \\
        &\leq c_4\, \wt C\bbbclr{ \sum_{i_1=1}^n \E\left[ X_{i_1} \right]
        +
         \sum_{ i_1, i_2=1}^n  \E\left[ X_{i_1}  \right] \E \left[X_{i_2} \right]
        \\
        &
        \qquad \qquad+
\sum_{i_1,i_2,i_3=1}^n\E\left[ X_{i_1} \right] \E \left[ X_{i_2} \right] \E \left[ X_{i_3} \right]
    +\sum_{i_1,i_2,i_3,i_4=1}^n \E\left[ X_{i_1}\right] \E \left[ X_{i_2}\right] \E \left[ X_{i_3}\right] \E \left[ X_{i_4} \right]}\\
        &  = c_4\, \wt C \bbbclr{ \sum_{j=1}^4 \E[X]^j},
    \end{align*}
    where the second inequality uses the hypothesis~\eqref{eq:negcor}. The result follows
    by taking $\wh C = 4 c_4\, \wt C$.
\end{proof}

To apply the proposition, we need to verify~\eqref{eq:negcor} for the Bernoulli sums in our setting, and have   good bounds on their means. For the latter we use the following result.

\begin{lemma}\label{lem:bounded first moments}
    Let $T_n \sim \UA(n)$, let $u \in \{0,\ldots,n\}$ and let $r \in \left[1,n\right]$. Then
    \begin{equation}\label{eq:bdd mom 1}
        \sum_{v=r}^{n} \p \left( v \sim u, F_{T_n}(v) > \frac{n}{r} \right) \leq 2,
    \end{equation}
    and
    \begin{equation}\label{eq:bdd mom 2}
        \sum_{v=u+1}^{r-1} \p \left( v \sim u, F_{T_n}(v) \leq \frac{n}{r} \right) \leq 1 .
    \end{equation}
\end{lemma}

\begin{proof}
    We start with a proof of \eqref{eq:bdd mom 1}. Noting that $(F_{T_n}(v)-1)$ is distributed as the number of times a red ball is chosen in a standard P\'olya urn started with $1$ red and $v$ blue balls after taking $n-v$ steps,  exchangeability implies
    \begin{equation*}
        \E \left[ F_{T_n}(v) \right] = 1 + \frac{n-v}{v+1} = \frac{n+1}{v+1} .
    \end{equation*}
    Since the events $\{v \sim u \}$ and $\left\{  F_{T_n}(v) > \frac{n}{r} \right\}$ are independent,  an application of Markov's inequality implies that
    \begin{align*}
         \sum_{v=r}^{n} \p \left( v \sim u, F_{T_n}(v) > \frac{n}{r} \right)
       &=
        \sum_{v=r}^{n} \p \left( v \sim u \right) \p \left( F_{T_n}(v) > \frac{n}{r} \right)
        \leq
        \sum_{v=r}^{n} \frac{1}{v} \frac{\E[F_{T_n}(v)]}{n/r} \\
        &        \leq  2 \sum_{v=r}^{n} \frac{r}{v(v+1)} = 2 r \left( \frac{1}{r} - \frac{1}{n+1} \right) \leq 2.
    \end{align*}
    For \eqref{eq:bdd mom 2}, noting again the P\'olya urn represpentation of  $F_{T_n}(v)-1$, a direct calculation gives 
    \begin{align*}
         \p \left( F_{T_n}(v) - 1 = k \right) 
        &=
        \binom{n-v}{k} \frac{\Gamma(1+k)\Gamma(n-k)}{\Gamma\left(1+n\right)} \frac{\Gamma\left(1+v\right)}{\Gamma(1)\Gamma(v)}
        \\
        &
        = \frac{(n-v)!k!(n-k-1)!v!}{k!(n-v-k)!n!(v-1)!} \\
        &= \frac{(n-v)!v}{n!} \frac{(n-1-k)!}{(n-v-k)!} .
    \end{align*}
    By considering consecutive ratios, the last expression is maximal when $k=0$. In particular, we have that for all $v\in \{1,\ldots, n\}$ and $\{x\in 1,\ldots, n\}$, 
    \begin{align*}
        \p \left( F_{T_n}(v) = x \right) \leq \p \left( F_{T_n}(v) - 1 = 0 \right) = \frac{(n-v)!v}{n!} \frac{(n-1)!}{(n-v)!} = \frac{v}{n}.
    \end{align*} 
    Summing over all possible values $x  \leq n/r$ yields
    \begin{align*}
        \p \left( F_{T_n}(v) \leq \frac{n}{r} \right) = \sum_{x=1}^{\frac{n}{r}} \p \left( F_{T_n}(v) = x \right) \leq  \sum_{x=1}^{\frac{n}{r}} \frac{v}{n} \leq \frac{v}{r}.
    \end{align*}
    Using again that the events $\left\{ v \sim u \right\}$ and $\left\{ F_{T_n}(v) \leq \frac{n}{r} \right\}$ are independent for $v \geq u+1$, we find that
    \begin{align*}
        & \sum_{v=u+1}^{r-1} \p \left( v \sim u, F_{T_n}(v) \leq \frac{n}{r} \right)
        =
        \sum_{v=u+1}^{r-1} \p \left( v \sim u \right) \p \left( F_{T_n}(v) \leq \frac{n}{r} \right)
        \leq
        \sum_{v=u+1}^{r-1} \frac{1}{v} \frac{v}{r} \leq 1 .\qedhere
    \end{align*}
\end{proof}

We can now prove the main result of this subsection.
\begin{proof}[Proof of Lemma~\ref{lem:zmombd}]
As previously mentioned, it is enough to show that fourth moment of each of the six random variables in the decomposition~\eqref{eq:zdecomp} of $Z_{\ell,j}^{T_n^i}(u)$ is uniformly bounded. 
We only show that $\E\bcls{ O_{\ell,j}^{T_n^i}(u)^4}$ and $\E\bcls{ \widetilde{O}_{\ell,j}^{T_n^i}(u)^4}$ are uniformly bounded; the terms $\E\bcls{ I_{\ell,j}^{1,T_n^i}(u)^4}$ and $\E\bcls{ \widetilde{I}_{\ell,j}^{1,T_n^i}(u)^4}$ can be treated like the term $\E\bcls{ O_{\ell,j}^{T_n^i}(u)^4}$, whereas the terms $\E\bcls{ I_{\ell,j}^{2,T_n^i}(u)^4}$ and $\E\bcls{ \widetilde{I}_{\ell,j}^{2,T_n^i}(u)^4}$ can be treated like the term $\E\bcls{ \widetilde{O}_{\ell,j}^{T_n^i}(u)^4}$. 

We start with the term $\E\bcls{ \widetilde{O}_{\ell,j}^{T_n^i}(u)^4}$, and  apply Proposition~\ref{claim:moment expansion}. To verify~\eqref{eq:negcor}, let $S\subset \left\{ \lceil e^\ell \rceil,\ldots,\lfloor e^j-1 \rfloor \right\}$ be a collection of vertices corresponding to the summands from $\widetilde{O}_{\ell,j}^{T_n^i}(u)$.
A key fact we will use that given $v\sim_i u$ for all $v\in S$,  
the process $\bclr{(F_{T_m^i}(v))_{v\in S}}_{m\geq \floor{e^j-1}}$  of sizes of the Fringe subtrees of vertices in $S$ starting from time $\floor{e^j-1}$ evolve like a multi-color P\'olya urn with random initialization. 
In particular, they are conditionally negatively related. 
Thus we have that
\begin{align*}
    \p \left(\bigcap_{v\in S}\left\{F_{T_n^i} (v) > \frac{n}{e^\ell}, v\sim_{i} u \right\}   \right) 
    &
    = 
    \p \left( \bigcap_{v\in S} \left\{ F_{T_n^i} (v) > \frac{n}{e^\ell} \right\} \Big| \bigcap_{v^\prime \in S} \left\{ v^\prime \sim_{i} u \right\} \right) \p \left( \bigcap_{v^\prime \in S} \left\{ v^\prime \sim_{i} u \right\} \right)
    \\
    &
    \leq 
    \left( \prod_{v\in S}
    \p \left(   F_{T_n^i} (v) > \frac{n}{e^\ell}  \Big| \bigcap_{v^\prime \in S} \left\{ v^\prime \sim_{i} u \right\} \right) \right) \p \left( \bigcap_{v^\prime \in S} \left\{ v^\prime \sim_{i} u \right\} \right)
    \\
    &
    = 
    \left( \prod_{v\in S}
    \p \left(   F_{T_n^i} (v) > \frac{n}{e^\ell}  \Big| \bigcap_{v^\prime \in S} \left\{ v^\prime \sim_{i} u \right\} \right) \right) \prod_{v^\prime \in S} \p \left(  v^\prime \sim_{i} u  \right)
    \\
    &
    \leq
    \left( \prod_{v\in S}
    \p \left(   F_{T_n^i} (v) > \frac{n}{e^\ell}  \Big|  v \sim_{i} u  \right) \right) \prod_{v^\prime \in S} \p \left(  v^\prime \sim_{i} u  \right)
    \\
    &
    =
    \prod_{v\in S}
    \p \left( v\sim_{i} u , F_{T_n^i} (v) > \frac{n}{e^\ell} \right),
\end{align*}
where the first inequality follows from the conditional negative relation already discussed, and the second inequality is because $\{v'\sim_i u\}$ is either independent of the event $\{F_{T_n^i}(v)>y\}$ (if $v'<v$), or is negatively related. 
Thus,~\eqref{eq:negcor} is satisfied with $\wt C=1$. To bound the first moment,~\eqref{eq:bdd mom 1} from Lemma~\ref{lem:bounded first moments} implies
 \be{
 \E \left[ \widetilde{O}_{\ell,j}^{T_n^i}(u) \right] = \sum_{v=e^\ell}^{e^j-1} \p \left( v\sim_{i} u , F_{T_n^i} (v) >  \frac{n}{e^\ell} \right) \leq \sum_{v=e^\ell}^{n} \p \left( v\sim_{i} u , F_{T_n^i} (v) >  \frac{n}{e^\ell} \right)\leq 2,
 }
and now finally applying Proposition~\ref{claim:moment expansion} yields the result for this case.

Next, we bound $\E\bcls{ O_{j,\ell}^{T_n^i}(u)^4}$ in an analogous way.
Let $S\subset \left\{ \lceil e^\ell \rceil,\ldots,\lfloor e^j-1 \rfloor \right\}$ be a collection of vertices, corresponding to the summands from $O_{j,\ell}^{T_n^i}(u)$, such that $|S| \leq 4$. Then 
\besn{\label{eq:product bound}
     &\p \left(\bigcap_{v\in S} \left\{ F_{T_n^i} (v) \leq \frac{n}{e^j}, v\sim_{i} u\right\} \right) 
    \\
    &\qquad
    = 
    \p \left( \bigcap_{v\in S} \left\{ F_{T_n^i} (v) \leq \frac{n}{e^j} \right\} \Big| \bigcap_{v^\prime \in S} \left\{ v^\prime \sim_{i} u \right\} \right) \p \left( \bigcap_{v^\prime \in S} \left\{ v^\prime \sim_{i} u \right\} \right)
    \\
    & \qquad
    \leq 
    \left( \prod_{v\in S}
    \p \left(   F_{T_n^i} (v) \leq \frac{n}{e^j}  \Big| \bigcap_{v^\prime \in S} \left\{ v^\prime \sim_{i} u \right\} \right) \right) \p \left( \bigcap_{v^\prime \in S} \left\{ v^\prime \sim_{i} u \right\} \right)
    \\
    &\qquad
    = 
    \left( \prod_{v\in S}
    \p \left(   F_{T_n^i} (v) \leq \frac{n}{e^j}  \Big| \bigcap_{v^\prime \in S \setminus \{v\}} \left\{ v^\prime \sim_{i} u \right\} \right) \right) \prod_{v^\prime \in S} \p \left(  v^\prime \sim_{i} u  \right),
}
where the inequality is due to conditional negative relation, and the second equality follows by direct calculation and independence, noting that $u\leq v'\in S$. 

To go further, the key observation is that the distribution of $F_{T_n^i}(v)$
conditional on the event $ \cap_{v^\prime \in S \setminus \{v\}} \left\{ v^\prime \sim_{i} u \right\}$ is stochastically larger than
the conditional distribution of $F_{T_n^i}(v)$ given that $v'$ does not join the fringe subtree of $v$ when it arrives, which we write as
$\cap_{v^\prime \in S \setminus \{v\}} \{ v^\prime \not\in \cF_{T_{v'}^i}(v) \}$, where $\cF_{T_m}^i(v)$ is the fringe subtree of $v$ at time $m$, and which we say is empty if $m<v$. 
This is because $(F_{T_m^i}(v))_{m\geq v}$ evolves as the number of red balls in a P\'olya urn started (at step $v$) with $1$ red ball and $v$ blue balls, and the first conditioning corresponds to saying draws $v'\in S$ with $v'>v$ result in a particular blue ball, whereas the latter conditioning corresponds to the same set of draws being any blue ball. It is intuitively clear that conditioning on choosing \emph{any} blue ball at some set of times has a stronger bias to smaller values of red balls than conditioning on choosing a \emph{particular} blue ball, and this yields the stochastic monotonicity. A formal proof using the correspondence to P\'olya urns is given in Lemma~\ref{lem:polstoch} below.

Now, because the size of the fringe subtree of $v$ evolves like the number of red balls in a P\'olya urn started with $1$ red ball and $v\geq1$ blue ball(s), exchangeability implies the chance vertex $v'$ joins the fringe subtree is $1/(1+v)\leq 1/2$. Combining this fact with the positive association of the events $\{ v^\prime \not\in \cF_{T_{v'}^i}(v) \}$ for $v'\in S\setminus \{v\}$, we have that
\ben{\label{eq:posassft}
\p\left( \bigcap_{v^\prime \in S \setminus \{v\}} \{ v^\prime \not\in \cF_{T_{v'}^i}(v) \}\right)
\geq \prod_{v^\prime \in S \setminus \{v\} }  \p\left( v^\prime \not\in \cF_{T_{v'}^i}(v) \right)\geq 2^{-\abs{S\setminus \{v\}}}\geq 2^{-3},
}
where in the last inequality we use that $\abs{S}\leq 4$.
Using the stochastic monotonicity and the inequality~\eqref{eq:posassft}, we can bound the terms appearing in the last line of~\eqref{eq:product bound} by
\ba{
\p \left(   F_{T_n^i} (v) \leq \frac{n}{e^j}  \Big| \bigcap_{v^\prime \in S \setminus \{v\}} \left\{ v^\prime \sim_{i} u \right\} \right)
   &\leq  \p \left(   F_{T_n^i} (v) \leq \frac{n}{e^j}  \Big| \bigcap_{v^\prime \in S \setminus \{v\}} \{ v^\prime \not\in \cF_{T_{v'}^i}(v) \}\right)\\
   &\leq  \frac{\p \left(   F_{T_n^i} (v) \leq \frac{n}{e^j}\right)}{ \p\left(  \bigcap_{v^\prime \in S \setminus \{v\}} \{ v^\prime \not\in \cF_{T_{v'}^i}(v) \}\right)} \\
   &\leq  8 \p \left( F_{T_n^i} (v) \leq \frac{n}{e^j} \right).
}
Inserting this into \eqref{eq:product bound} and using again that $|S|\leq 4$, we get that
\begin{align*}
      \p \left(\bigcap_{v\in S} \left\{ F_{T_n^i} (v) \leq \frac{n}{e^j}, v\sim_{i} u\right\} \right) 
    &
    \leq
    \left( \prod_{v\in S}
    \p \left(   F_{T_n^i} (v) \leq \frac{n}{e^j}  \Big| \bigcap_{v^\prime \in S \setminus \{v\}} \left\{ v^\prime \sim_{i} u \right\} \right) \right) \prod_{v^\prime \in S} \p \left(  v^\prime \sim_{i} u  \right)
    \\
    &
    \leq
    \left( \prod_{v\in S}
    8 \p \left( F_{T_n^i} (v) \leq \frac{n}{e^j} \right) \right) \prod_{v^\prime \in S} \p \left(  v^\prime \sim_{i} u  \right)
    \\
    & 
    = 8^{|S|}
    \left( \prod_{v\in S}
    \p \left( F_{T_n^i} (v) \leq \frac{n}{e^j} \Big| v \sim_{i} u \right) \right) \prod_{v^\prime \in S} \p \left(  v^\prime \sim_{i} u  \right)
    \\
    &
    \leq 8^{4} \prod_{v\in S}
    \p \left( F_{T_n^i} (v) \leq \frac{n}{e^j} , v \sim_{i} u \right),
\end{align*}
verifying~\eqref{eq:negcor} with~$\wt C=8^{4}$.
The first moment is bounded by~\eqref{eq:bdd mom 2} from Lemma~\ref{lem:bounded first moments}, which gives
\be{
\E \left[ O_{\ell,j}^{T_n^i}(u) \right] = \sum_{v=e^\ell}^{e^j-1} \p \left( v\sim_{i} u , F_{T_n^i} (v) \leq  \frac{n}{e^j} \right) \leq \sum_{v=u+1}^{e^j-1} \p \left( v\sim_{i} u , F_{T_n^i} (v) \leq  \frac{n}{e^j} \right)\leq 1.
}
The result  now follows from  Proposition~\ref{claim:moment expansion}. 
\end{proof}

\section{P\'olya urn lemmas}\label{sec:polbd}
Here we record some results about
standard P\'olya urns used above. 
\begin{lemma}\label{lem:polstoch}
    Let $(R_m)_{m\geq0}$ be the sequence of counts of the number of red balls in a standard P\'olya urn started with $r\in \N$ red balls and $b\in \N$ blue balls. Assume also that one of the blue balls is in the urn initially is `distinguished' with the label $\mathfrak{d}$. 
    Let $\cB_\ell$ denote the event that a blue ball is chosen at step $\ell$, and $\cD_\ell$ denote the event that the the ball marked $\mathfrak{d}$ is chosen at step $\ell$. 
    Then for any sequence of draws $0<m_1<\cdots< m_k \leq m$, 
    \ben{ \label{lem5.1 result}
    \mathscr{L}\bclr{R_m |\cap_{i=1}^k \cB_{m_i}} \leq_{st} \mathscr{L}\bclr{R_m | \cap_{i=1}^k \cD_{m_i}}.
    }
\end{lemma}
\begin{proof}
    We claim that 
    for any bounded function $g$,
    \ben{\label{eq:com}
    \E\bcls{g(R_m) | \cap_{i=1}^k \cB_{m_i} } = \E\bcls{g(R_m)  p(R_{m_1},\ldots, R_{m_k}) | \cap_{i=1}^k \cD_{m_i} },
    }
    where the change of measure function $p$ is defined by
    \be{ 
     p(R_{m_1},\ldots, R_{m_k})= \P\bclr{\cap_{i=1}^k \cB_{m_i}}^{-1} \prod_{i=1}^k \bbclr{1- \frac{R_{m_i-1}}{r+b + m_i -1}}.
    }
This change of measure formula follows by considering any sequence of draws $n_j\in\{\mathrm{red}, \mathrm{blue}\}$ for $j=1,\ldots,m$ with $j\not\in\{m_1,\ldots, m_k\}$. Here $n_j=\mathrm{red}$ ($\mathrm{blue}$) means that a red (blue) ball is drawn at step $n_j$.  Setting $\cR_m:= \{ j: n_j = \mathrm{red}\}$ and  $r_\ell = |\{j\leq \ell: n_j=\mathrm{red}\}|$, the chance of observing this sequence and $\cap_{i=1}^k \cB_i$ is given by
     \be{
     \frac{ \prod_{\ell\in \cR_m} r_{\ell-1} \times \prod_{\ell \in \{1,\ldots,m\}\setminus \cR_m } (r+b + \ell-1 -r_{\ell-1})  }{\prod_{j=0}^{m-1} (r+b + j)},
     }
     whereas  the chance of observing this sequence and $\cap_{i=1}^k \cD_i$ is given by
     \be{
     \frac{ \prod_{\ell\in \cR_m} r_{\ell-1} \times \prod_{\substack{\ell \in \{1,\ldots,m\}\setminus \cR_m \\ \ell \not\in \{m_1, \ldots, m_k\}} } (r+b + \ell-1 -r_{\ell-1})  }{\prod_{j=0}^{m-1} (r+b + j)}.
     }
     Noting that $\P(\cap_{i=1}^k \cD_i) = 1/\prod_{i=1}^k (r+b+m_i-1)$  yields the claim.
     
Next, we use induction to prove that for any two non-increasing functions $f,g:\R^{\ell} \to \R$, any sequence of draws $0<m_1<\ldots<m_\ell$, and any set $A\subset \N$, we have 
\begin{align}
    &\notag \E\bcls{f( R_{m_1}, \ldots , R_{m_\ell}) g( R_{m_1}, \ldots , R_{m_\ell})   | \cap_{a \in A} \cD_{a} } \\
    &\hspace{3cm} \label{eq:positv associt}
    \geq    \E\bcls{f( R_{m_1}, \ldots , R_{m_\ell})   | \cap_{a \in A} \cD_{a} } 
     \E\bcls{g( R_{m_1}, \ldots , R_{m_\ell})   | \cap_{a \in A} \cD_{a} } .
\end{align}
The statement is true in general for $\ell=1$. For the induction step from $\ell$ to $\ell+1$, the law of total expectation gives that
\begin{align*}
    & \sum_{t=1}^{\infty} \E\bcls{f( R_{m_1}, \ldots , R_{m_{\ell+1}}) g( R_{m_1}, \ldots , R_{m_\ell}) | R_{m_1}= t, \cap_{a \in A} \cD_{a} } \p \left( R_{m_1}= t | \cap_{a \in A} \cD_{a} \right). 
\end{align*}
Conditioned on $R_{m_1}= t$ (and any other events involving times $m\leq m_1$), the process $(R_m)_{m\geq m_1}$ evolves like a standard P\'olya urn started with $t$ red balls and $r+b+m_1-t$ blue at time $m_1$, so that the induction hypothesis implies that
\begin{align}
    & \notag \E\bcls{f( R_{m_1}, \ldots , R_{m_{\ell+1}}) g( R_{m_1}, \ldots , R_{m_\ell}) | R_{m_1}= t, \cap_{a \in A} \cD_{a} } 
    \\
    & \notag
    =
    \E\bcls{f( t,R_{m_2}, \ldots , R_{m_{\ell+1}}) g( t, R_{m_2}, \ldots , R_{m_\ell}) | R_{m_1}= t, \cap_{a \in A} \cD_{a} } 
    \\
    & \label{inequality 1}
    \geq
    \E\bcls{f( t, R_{m_2}, \ldots , R_{m_{\ell+1}}) | R_{m_1}= t, \cap_{a \in A} \cD_{a} }
    \E\bcls{g( t, R_{m_2}, \ldots , R_{m_{\ell+1}}) | R_{m_1}= t, \cap_{a \in A} \cD_{a} }.
\end{align}
Define the functions $F$ and $G$ by
\begin{align*}
    & F(t) := \E\bcls{f( t ,R_{m_2}, \ldots , R_{m_{\ell+1}}) | R_{m_1}= t, \cap_{a \in A} \cD_{a} }
    , \text{ and }
    \\
    &
    G(t) := \E\bcls{g( t, R_{m_2}, \ldots , R_{m_{\ell+1}}) | R_{m_1}= t, \cap_{a \in A} \cD_{a} } .
\end{align*}
There exists a straight-forward coupling of $(R_m)_{m\geq m_1}$ of the measures $\p \left( \cdot | R_{m_1}= t, \cap_{a \in A} \cD_{a} \right)$ so that each random variable $(R_m)_{m\geq m_1}$ is increasing in $R_{m_1}=t$. Thus, we see that the functions $F$ and $G$ are non-increasing in $t$. Combining this with \eqref{inequality 1}, we see that
\begin{align*}
    \sum_{t=1}^{\infty}  \E\bcls{f( R_{m_1}, \ldots , R_{m_{\ell+1}}) & g( R_{m_1}, \ldots , R_{m_\ell}) | R_{m_1}= t, \cap_{a \in A} \cD_{a} } \p \left( R_{m_1}= t | \cap_{a \in A} \cD_{a} \right)
    \\
    &
    \geq
    \sum_{t=1}^{\infty} F(t) G(t) \p \left( R_{m_1}= t | \cap_{a \in A} \cD_{a} \right)\\
    &=
    \E \left[ F(R_{m_1}) G(R_{m_1}) | \cap_{a \in A} \cD_{a}  \right]
    \\
    &
    \geq
    \E \left[ F(R_{m_1})  | \cap_{a \in A} \cD_{a}  \right]
    \E \left[ G(R_{m_1}) | \cap_{a \in A} \cD_{a}  \right].
\end{align*}
The induction step follows, since
\begin{align*}
    \E \left[ F(R_{m_1})  | \cap_{a \in A} \cD_{a}  \right] & = \sum_{t=1}^{\infty} \E\bcls{f( t ,R_{m_2}, \ldots , R_{m_{\ell+1}}) | R_{m_1}= t, \cap_{a \in A} \cD_{a} } \p \left(  R_{m_1}= t | \cap_{a \in A} \cD_{a}  \right)
    \\
    &
    =
     \E\bcls{f( R_{m_1} ,R_{m_2}, \ldots , R_{m_{\ell+1}}) |  \cap_{a \in A} \cD_{a} },
\end{align*}
and, analogously,
\begin{align*}
    \E \left[ G(R_{m_1})  | \cap_{a \in A} \cD_{a}  \right]=
     \E\bcls{g( R_{m_1} ,R_{m_2}, \ldots , R_{m_{\ell+1}}) |  \cap_{a \in A} \cD_{a} }.
\end{align*}

The stochastic domination, \eqref{lem5.1 result}, now follows by considering a non-increasing function $g$, and writing $\phi_1( R_{m_1}, \ldots , R_{m_k}, R_m):= g(R_m)$, $\phi_2(R_{m_1}, \ldots , R_{m_k}, R_m) := p(R_{m_1}, \ldots , R_{m_k})$, which are both non-increasing in their arguments. Then 
 \ba{
  \E\bcls{g(R_m) | \cap_{i=1}^k \cB_{m_i} } 
    &= \E\bcls{\phi_1( R_{m_1}, \ldots , R_{m_k}, R_m)\phi_2( R_{m_1}, \ldots , R_{m_k}, R_m)   | \cap_{i=1}^k \cD_{m_i} } \\
    &\geq    \E\bcls{\phi_1( R_{m_1}, \ldots , R_{m_k}, R_m)   | \cap_{i=1}^k \cD_{m_i} } \E\bcls{\phi_2( R_{m_1}, \ldots , R_{m_k}, R_m)   | \cap_{i=1}^k \cD_{m_i} } \\
      &=  \E\bcls{g(R_m)| \cap_{i=1}^k \cD_{m_i} },  
 }
 where the first equality is the change of measure formula~\eqref{eq:com}, the inequality follows from \eqref{eq:positv associt}, and the last equality is from the definition of $\phi_1$, and the change of measure formula~\eqref{eq:com} with $g\equiv 1$. This proves \eqref{lem5.1 result}. 
\end{proof}

The next lemma gives a close coupling of the marginal
distribution of a standard P\'olya urn with a variable distributed as the limiting beta of the urn. It is essentially~\citet*[Lemma~3]{Pekoz2017}.
\begin{lemma}\label{lem:purntbd}[Lemma~3 of \cite{Pekoz2017}]
Let $R_m$ be the number of red balls in a standard P\'olya urn  started with $r\in\mathbb{N}$ red balls and $b\in \mathbb{N}$ blue balls after $m$ draws and replacements. Then there is a coupling of $R_m$ and  $B_{r,b}\sim\mathrm{Beta}(r,b)$ such that
\[
|R_m - m B_{r,b}| \leq 2(b\wedge r) (r+b), 
\]
and consequently
\begin{align*}
	\P(R_m \geq t) &\geq  \P(m B_{r,b} \geq t + 2(b\wedge r) (r+b)).
\end{align*}
\end{lemma}
\begin{proof}
The first assertion follows easily from \cite[Lemma~3]{Pekoz2017} which says there is a coupling of $R_m$ with $B=B_{r,b}\sim\mathrm{Beta}(r,b)$ such that
\[
| R_m - mB | \leq 2 b (r+b).
\]
Reversing the roles of $r$ and $b$ and noting that $m-R_m$ counts the number of blue balls and $(1-B) \sim \mathrm{Beta}(b,r)$, we can reverse the roles of
$r$ and $b$ and get a coupling of $(m-R_m)$ and $(1-B)$ such that 
\[
|R_m - m B| = |(m-R_m)- m(1-B)|  \leq 2 r (r+b).
\]
Choosing the appropriate coupling based on which of $r$ or $b$ is smaller, implies there is a coupling of $R_m$ and $B$ such that,
\[
| R_m - mB| \leq 2 (b \wedge r) (r + b).
\]
The consequence now easily follows by using the coupling to find 
\[
\P(R_m \geq t) = \P(m B \geq t - (R_m - m B) ) \geq \P(mB \geq t+ 2(b\wedge r) (r+b)). \qedhere
\]
\end{proof}


\section*{Acknowledgements}

We thank Nicol\'as Agote, Shankar Bhamidi, Luc Devroye, Igor Kortchemski, G\'abor Lugosi, and Lucas Prates for valuable discussions. 
This material is based upon work supported in part by the National Science Foundation (NSF) under Grant No.~DMS-1928930, while M.Z.R. was in residence at the Simons Laufer Mathematical Sciences Institute (SLMath/MSRI) in Berkeley, California, participating in the Probability and Statistics of Discrete Structures program during Spring~2025.


\bibliographystyle{abbrvnat}
\bibliography{references}


\newpage
\appendix

\section{Counting leaves} \label{sec:leaves} 

For $\alpha \in [0,1]$, let $(T_n^1, T_n^2) \sim \alphaCUA(n)$ be a pair of correlated, unlabeled trees, 
and let $L_1(n)$ and $L_2(n)$ denote the fraction of leaves in $T_n^1$ and $T_n^2$, respectively. 
In this section we show that counting leaves in the two trees allows to distinguish between independent trees and correlated trees with positive power. Specifically, we show the following result. 
\begin{theorem}\label{thm:leaves}
For $\alpha \in [0,1]$, let 
$\sigma_{\alpha}^{2} := \frac{(1 - \alpha)(3 + \alpha)}{6 (3 - \alpha)}$. 
For every $\alpha \in (0,1)$ we have that
\[
\liminf\limits_{n \to \infty} \TV \left( \UA(n)^{\otimes 2}, \alphaCUA(n) \right) \ge \TV \left( \cN \left(0, 1/6 \right), \cN\left( 0, \sigma_\alpha^2 \right) \right) > 0,
\]
where $\cN(\mu, \sigma^{2})$ denotes a Gaussian random variable with mean $\mu$ and variance $\sigma^{2}$.
\end{theorem}
We provide an overview of the analysis here, and leave the details to the interested~reader. 
Define the sets
\begin{align*}
\cS_{11}(n) & : = \left \{0 \le v \le n : \deg_{T_n^1}(v) = \deg_{T_n^2} (v) = 1 \right \}, \\
\cS_{1 \times }(n) & : = \left \{ 0 \le v \le n: \deg_{T_n^1}(v) = 1, \deg_{T_n^2}(v) \ge 2 \right \}, \\
\cS_{\times 1} (n) & : = \left \{ 0 \le v \le n: \deg_{T_n^1}(v) \ge 2, \deg_{T_n^2}(v) = 1 \right \}, \\
\cS_{\times \times}(n) & : = \left \{ 0 \le v \le n: \deg_{T_n^1}(v) \ge 2, \deg_{T_n^2}(v) \ge 2 \right\}.
\end{align*}
Note that these sets partition the set $\{0, 1, \ldots, n \}$, while also capturing similarities in leaves with the same node label across $T_n^1$ and $T_n^2$. 
In particular, the difference between the number of leaves in the two trees is precisely captured by 
$| \cS_{1 \times}(n)| - |\cS_{\times 1}(n)| = (n+1)(L_{1}(n) - L_{2}(n))$. 

\paragraph{Deriving stochastic recursions.} The first step of the analysis is to characterize the evolution of the set sizes $| \cS_{11}(n) |, | \cS_{1 \times}(n) |$ and $| \cS_{\times 1}(n)|, | \cS_{\times \times}(n) |$, which depend on the sequence of aligned and independent choice steps in the joint construction of the two trees. 

As an example, consider the evolution of $| \cS_{11}(n)|$.  
We first note that when vertex $n + 1$ is added, it is clearly a leaf in both trees; thus $n + 1 \in \cS_{11}(n + 1)$. Depending on the type of choice (aligned or independent) and the specific connections made by vertex $n + 1$, some vertices in $\cS_{11}(n)$ may not be present in $\cS_{11}(n + 1)$. We outline the possibilities below: 
\begin{itemize}
    \item {\bf Aligned choice.} If $n + 1$ attaches to vertex $u$ in both trees and $u \in \cS_{11}(n)$, then $u \notin \cS_{11}(n + 1)$. Otherwise, all vertices in $\cS_{11}(n)$ are also present in $\cS_{11}(n + 1)$. Hence
    \[
    | \cS_{11}(n + 1) | - | \cS_{11}(n) | = 1 - \mathbf{1}(u \in \cS_{11}(n) ).
    \]
    
    \item {\bf Independent choice.} Suppose that $n + 1$ attaches to $u^1$ in $T_n^1$ and $u^2$ in $T_n^2$. If $u^1 \in \cS_{11}(n)$, then $u^1 \notin \cS_{11}(n + 1)$. The same logic holds for $u^2$. 
    Consequently, 
    \[
    | \cS_{11}(n + 1) | - | \cS_{11}(n) | = 1 - \mathbf{1}( u^1 \in \cS_{11}(n) ) - \mathbf{1}( u^2 \in \cS_{11}(n)) + \mathbf{1}(u^1 = u^2 \text{ and } u^1 \in \cS_{11}(n) ),
    \]
    where the addition of the third indicator is due to double-counting by the first two indicators when $u^1 = u^2$.
\end{itemize}
Putting everything together, we obtain the recursion
\begin{align*}
\E \left[ | \cS_{11}(n + 1) | \vert \cF_n \right] & = | \cS_{11}(n) | + 1 - \frac{\alpha | \cS_{11}(n)|}{n+1} - (1 - \alpha) \left( \frac{2 | \cS_{11}(n)|}{n+1} - \frac{ | \cS_{11}(n) |}{(n+1)^2} \right) \\
& = | \cS_{11}(n) | + 1 - (2 - \alpha) \frac{ | \cS_{11}(n)|}{n+1} + (1 - \alpha) \frac{| \cS_{11}(n) |}{(n+1)^2}.
\end{align*}
The normalized set size evolves as
\begin{equation}
\label{eq:S11_recursion}
\E \left[ \left. \frac{ | \cS_{11}(n + 1)|}{n + 2} \right \vert \cF_n \right] = \frac{ |\cS_{11}(n) |}{n+1} + \frac{1}{n + 2} \left(1 - (3 - \alpha) \frac{| \cS_{11}(n) |}{n+1} + (1 - \alpha) \frac{ | \cS_{11}(n) |}{(n+  1)^2} \right),
\end{equation}
and similar arguments give the following recursions for $| \cS_{1 \times}(n) |$ and $| \cS_{\times 1}(n)|$, 
\besn{\label{eq:S1x_recursion}
\E \bbbcls{  &\left. \frac{ | \cS_{1 \times}(n + 1)|}{n + 2} \right  \vert \cF_n } \\
&= \frac{| \cS_{1 \times}(n) |}{n+1} + \frac{1}{n + 2} \left( - \frac{2 | \cS_{1 \times}(n) |}{n+1} + (1 - \alpha) \frac{ | \cS_{11}(n) |}{n+1} - (1 - \alpha) \frac{ | \cS_{11}(n)|}{(n+1)^2} \right),
}
\besn{
\label{eq:Sx1_recursion}
\E \bbbcls{ & \left. \frac{ | \cS_{\times 1}(n + 1)|}{n + 2} \right \vert \cF_n } \\
    &= \frac{| \cS_{\times 1}(n) |}{n+1} + \frac{1}{n + 2} \left( - \frac{2 | \cS_{\times 1}(n) |}{n+1} + (1 - \alpha) \frac{ | \cS_{11}(n) |}{n+1} - (1 - \alpha) \frac{ | \cS_{11}(n)|}{(n+1)^2} \right).
}

\paragraph{Asymptotic normality.} Well-established results on stochastic approximation such as \cite{Fabian1968} and \cite{Hernandez2019} allow us to characterize the limiting distribution of the iterates $| \cS_{11}(n) | / n, | \cS_{1 \times}(n) | / n$, and $| \cS_{\times 1}(n) | / n$. 
By applying known results to the coupled stochastic recursion \eqref{eq:S11_recursion}, \eqref{eq:S1x_recursion}, and \eqref{eq:Sx1_recursion}, we obtain that 
\begin{equation}
\label{eq:asymptotic_normality}
\sqrt{n} \begin{pmatrix}
\frac{| \cS_{1 \times}(n) |}{n} - \frac{1 - \alpha}{2(3 - \alpha)} \\
\frac{| \cS_{\times 1}(n) |}{n} - \frac{1 - \alpha}{2 (3 - \alpha)} \\
\frac{| \cS_{11}(n) |}{n} - \frac{1}{3 - \alpha} 
\end{pmatrix} 
\stackrel{d}{\to}
\cN \left(0, Q \right),
\end{equation}
where $Q$ is a $3 \times 3$ matrix with entries given by 
\begin{align*}
Q_{33} & =  \frac{4 - 3 \alpha}{(3 - \alpha)^2 (5 - 2 \alpha)} \\
Q_{23} = Q_{32} &  = \frac{ (1 - \alpha)(4 \alpha - 17) }{2 (3 - \alpha)^2 (5 - 2 \alpha) (4 - \alpha) }\\
Q_{22} & =   - \frac{ (1 - \alpha) (2\alpha^3 - 23\alpha^2 + 105 \alpha - 192 )}{12 (4 - \alpha) (3 - \alpha)^2 (5 - 2 \alpha) } \\
Q_{13} = Q_{31} &  = \frac{ (1 - \alpha)(4 \alpha - 17) }{2 (3 - \alpha)^2 (5 - 2 \alpha) (4 - \alpha) } \\
Q_{12} = Q_{21} &  = \frac{ (1 - \alpha) (2\alpha^4 - 15\alpha^3 + 25 \alpha^2 + 12 \alpha + 12) }{ 12 (4 - \alpha) (3 - \alpha)^2 (5 - 2 \alpha) } \\
Q_{11} & =  - \frac{(1 - \alpha)(2\alpha^3 - 23 \alpha^2 + 105 \alpha - 192)}{12 (4 - \alpha) (3 - \alpha)^2 (5 - 2\alpha) }.
\end{align*}
The form of the matrix $Q$ is computed based on the stochastic approximation results of \cite{Hernandez2019}. At a high level, the calculation first involves computing the covariance matrix of the martingale differences corresponding to the stochastic recursion of interest, and applying a certain recursive, entry-wise transformation to compute $Q$.

We make a few remarks about the asymptotic normality result \eqref{eq:asymptotic_normality} and its relation to known results. 
\begin{itemize}
    \item When $\alpha = 1$, the two jointly generated trees are identical, and the number of leaves in each tree is captured by $\cS_{11}(n)$. The display in~\eqref{eq:asymptotic_normality} shows that $| \cS_{11}| / n$ concentrates around $1/2$ in this case, with asymptotically normal fluctuations of variance $Q_{33} = 1/12$. This recovers known distributional results on the number of leaves in a single UA tree; see \cite{Janson2005}.

    \item For any $\alpha \in [0,1]$, 
    the fraction of leaves in $T_n^1$ is equal to $L_{1}(n) = ( | \cS_{11}(n) | + | \cS_{1 \times}(n) | ) / n$. 
    According to~\eqref{eq:asymptotic_normality}, this concentrates around $1/2$ with asymptotically normal fluctuations of variance $Q_{33} + Q_{11} + 2 Q_{13} = 1/12$. Hence the marginal characterization of leaves in a single UA tree aligns with the results of \cite{Janson2005}.
\end{itemize}

\paragraph{Distributional differences in the fraction of leaves.}
The convergence in~\eqref{eq:asymptotic_normality} implies that 
\begin{equation}
\label{eq:normal_stat}
\sqrt{n} ( L_1(n) - L_2(n)) \stackrel{d}{\to} \cN \left( 0, \frac{(1 - \alpha)(3 + \alpha)}{6 (3 - \alpha)} \right).
\end{equation}
Thus, we can use the difference between the fraction of leaves in the two trees as a differentiating statistic. 
Specifically, 
if $(T_n^1, T_n^2) \sim \UA(n)^{\otimes 2}$, then 
$\sqrt{n} ( L_1(n) - L_2(n)) \stackrel{d}{\to} \cN \left( 0, 1/6 \right)$, 
while if $(T_n^1, T_n^2) \sim \alphaCUA(n)$, then 
$\sqrt{n} ( L_1(n) - L_2(n)) \stackrel{d}{\to} \cN \left( 0, \sigma_\alpha^2 \right)$, 
where recall that 
$\sigma_{\alpha}^{2} := \frac{(1 - \alpha)(3 + \alpha)}{6 (3 - \alpha)}$. 
This proves Theorem~\ref{thm:leaves} by noting that $\sigma_{\alpha}^{2} \neq 1/6$ for $\alpha > 0$. 

Since $\sigma_\alpha \to 0$ as $\alpha \to 1$, it is straightforward to see that $\TV \left( \UA(n)^{\otimes 2}, \alphaCUA(n) \right)\to 1$ as $\alpha \to 1$ and $n \to \infty$. For example, if $X_\alpha=\sigma_\alpha Z$ and $Y= \sqrt{1/6} Z$ with $Z\sim \cN(0 , 1)$, then as $\alpha \to 0$, we have  $\P(\abs{X_\alpha} > \sqrt{\sigma_\alpha})=\P(\abs{Z}> 1/ \sqrt{\sigma_\alpha}) \to 0$ and $\P(\abs{Y} > \sqrt{\sigma_\alpha})=\P(\abs{Z} > \sqrt{6 \sigma_\alpha})\to 1$.


\paragraph{Beyond leaves.} A natural question is whether comparing the fraction of higher degree vertices across the two trees can allow us to distinguish $\UA(n)^{\otimes 2}$ and $\alphaCUA(n)$ with power close to 1. While we expect that asymptotic normality can be established for joint degree counts in a similar sense as \eqref{eq:asymptotic_normality}, the corresponding more general covariance matrix $Q$ does not seem to admit a tractable representation for analysis.
Even if there is no closed form for $Q$, understanding its qualitative properties would shed insight into the joint structure of correlated UA trees and may also lead to better algorithms for identifying and estimating correlated regions in the two trees. 

\end{document}